\documentclass[11pt]{amsart}

\usepackage{color, amsmath,amssymb, amsfonts, amstext,amsthm, latexsym}

   \usepackage{latexsym, oldlfont}

   \newtheorem{lemma}{Lemma}[section]
   \newtheorem{theorem}[lemma]{Theorem}
   
   \newtheorem{prop}[lemma]{Proposition}

\newcommand{\R}{{\mathbb R}}

\newcommand{\beq}{\begin{equation}}
\newcommand{\eeq}{\end{equation}}
\newcommand{\Del}{\Delta}

\newcommand{\nab}{\nabla}
\newcommand{\lam}{\lambda}

\newcommand{\ome}{\omega}
\newcommand{\patl}{\partial}

\title[Navier-Stokes Equations]{Cylindrical Symplectic Representation
and Global Regular Solution of Incompressible Navier-Stokes Equations
in $\R^3$}

\author{Yongqian Han}

\date{}

\address{
Institute of Applied Physics and Computational Mathematics,
Beijing 100088, China}
\address{
National Key Laboratory of Computational Physics,
Beijing 100088, China}
\email{han\_yongqian@iapcm.ac.cn}



\begin{document}

\begin{abstract} The existence and uniqueness of global regular
solution of incompressible Navier-Stokes equations in $\R^3$ are
derived provided the initial velocity vector field holds a special
structure.
\end{abstract}

\maketitle

~ ~ ~ ~

\section{Introduction}
\setcounter{equation}{0}

The Navier-Stokes equations in $\R^3$ with initial data are given by
\beq\label{NS1}
u_t-\nu\Del u+(u\cdot\nab)u+\nab P=0,
\eeq
\beq\label{NS2}
\nab\cdot u=0,
\eeq
\beq\label{NSi}
u|_{t=0}=u_0,
\eeq
where $u=u(t,x)=\big(u^1(t,x),u^2(t,x),u^3(t,x)\big)$ and $P=P(t,x)$
stand for the unknown velocity vector field of fluid and its
pressure, $u_0=u_0(x)=\big(u^1_0(x),u^2_0(x),u^3_0(x)\big)$ is
the given initial velocity vector field satisfying $\nab\cdot u_0=0$,
$\nu>0$ is the coefficient of viscosity. Here $\patl_{x_j}$ denotes
by $\patl_j$ ($j=1,2,3$).

For the mathematical setting of this problem, we introduce
Hilbert space
\[
H(\R^3)=\big\{u\in\big(L^2(\R^3)\big)^3\big|\nab \cdot u=0\big\}
\]
endowed with $\big(L^2(\R^3)\big)^3$ norm (resp. scalar product
denoted by $(\cdot,\cdot)$ ). For simplicity of presentation, space
$\big(H^m(\R^3)\big)^3$ denotes by $H^m$, where $m\ge0$. In what
follows we use the usual convention and sum over the repeated indices.

The Fourier transformation of $u(t,x)$ with respect to $x$ denotes by
$\hat{u}(t,\xi)$. Then equation \eqref{NS2} can be rewritten as follows
\beq\label{NS2-F}
\xi_j\hat{u}^j=\xi_1\hat{u}^1+\xi_2\hat{u}^2+\xi_3\hat{u}^3=0.
\eeq
It is that $\hat{u}(t,\xi)$ is perpendicular to $\xi$. Denote by $\xi
\bot\hat{u}$. Equivalently $\hat{u}(t,\xi)\in T_{\xi}{\mathbb{S}}^2$.

Let $A=(a,b,c)\in\R^3-\{0\}$ and $\xi\in{\mathbb{S}}^2$. The vector
$\{A\times\xi\}\in T_{\xi}{\mathbb{S}}^2$ is called the 1st order
incompressible symplectic symmetry. Take $3\times3$ matrix $M=\big(
m^{ij}\big)$. The vector $\{\xi\times(M\xi)\}\in T_{\xi}{\mathbb{S}}
^2$ is called the 2nd order incompressible symplectic symmetry.
Generally let $T=\big(T^{i_1\cdots i_n}\big)$ be $n$th order tensor
and $T(\xi\cdots\xi)=T^{i_1i_2\cdots i_n}\xi_{i_2}\cdots\xi_{i_n}$.
The vector $\{\xi\times T(\xi\cdots\xi)\}\in T_{\xi}{\mathbb{S}}^2$
is called the $n$th order incompressible symplectic symmetry.

Let vectors $A\in\R^3-\{0\}$ and $B\in\R^3-\{0\}$ be linear
independent. Then $A\times\xi$ and $B\times\xi$ are linear independent
for any $\xi\in\R^3-\{0\}$, and form the basis of space $T_{\xi}
{\mathbb{S}}^2$ which is so-called moving frame. There exist $
\hat{\phi}(\xi)$ and $\hat{\psi}(\xi)$ such that
\[
\hat{u}(\xi)=\hat{\phi}(\xi)\cdot\{A\times\xi\}+\hat{\psi}(\xi)\cdot
\{B\times\xi\},\;\;\forall \hat{u}(\xi)\in T_{\xi}{\mathbb{S}}^2.
\]

For any velocity vector $u$ satisfying the equation\eqref{NS2}, there
exist real scalar functions $\phi$ and $\psi$ such that
\beq\label{CSR12-R3}\begin{split}
u(t,x)=&(A\times\nab)\phi(t,x)+\{(A\times\nab)\times\nab\}\psi(t,x),\\
\end{split}\eeq
where vector $A=(a_1,a_2,a_3)\in\R^3-\{0\}$. The formulation
\eqref{CSR12-R3} is called (1,2)-symplectic representation of the
velocity vector $u$.

Let
\beq\label{Def-CSR12-ome}\begin{split}
\ome(t,x)=&\nab\times u(t,x)\\
=&-\{(A\times\nab)\times\nab\}\phi(t,x)+(A\times\nab)\Del\psi(t,x).\\
\end{split}\eeq
Taking curl with equation \eqref{NS1}, we drive
\beq\label{NS-CSR12-ome}
\ome_t-\nu\Del\ome+(u\cdot\nab)\ome-(\ome\cdot\nab)u=0.
\eeq

Thanks the following observations
\beq\label{CSR12-phi}\begin{split}
(A\times\nab)\cdot u(t,x)&=(A\times\nab)\cdot(A\times\nab)\phi(t,x)\\
&=\big\{(A\cdot A)\Del-(A\cdot\nab)^2\big\}\phi(t,x),\\
\end{split}\eeq
\beq\label{CSR12-psi}\begin{split}
(A\times\nab)\cdot \ome(t,x)&=(A\times\nab)\cdot(A\times\nab)\Del\psi(t,x)\\
&=\big\{(A\cdot A)\Del-(A\cdot\nab)^2\big\}\Del\psi(t,x),\\
\end{split}\eeq
taking scalar product of equation \eqref{NS1} with $A\times\nab$,
we obtain
\beq\label{NS-CSR12-phi0}
\big\{(A\cdot A)\Del-(A\cdot\nab)^2\big\}\{\phi_t-\nu\Del\phi\}
+(A\times\nab)\cdot\{(u\cdot\nab)u\}=0.
\eeq
And taking scalar product of equation \eqref{NS-CSR12-ome} with
$A\times\nab$, we get
\beq\label{NS-CSR12-psi0}
\big\{(A\cdot A)\Del-(A\cdot\nab)^2\big\}\Del\{\psi_t-\nu\Del\psi\}
+(A\times\nab)\cdot\{(u\cdot\nab)\ome-(\ome\cdot\nab)u\}=0.
\eeq

There is a large literature studying the incompressible Navier-Stokes
equations. In 1934 Leray \cite{Le34} proved that there exists a global
weak solution to the problem \eqref{NS1}--\eqref{NSi} with initial data
in $L^2$. In 1951 Hopf \cite{Hop51} extended this result to bounded
smooth domain. Moreover Leray-Hopf weak solutions satisfy energy
inequality \cite{Te01}
\beq\label{En-Ineq}
\|u(t,\cdot)\|^2_{L^2}+2\int^t_0\|\nab u(\tau,\cdot)\|^2_{L^2}d\tau
\le \|u_0\|^2_{L^2},\;\;\;\;\forall t>0.
\eeq

The uniqueness and regularity of Leray-Hopf weak solution is a famous
open question. Numerous regularity criteria were proved \cite{ESS03,
Gig86,KeK11,KNSS09,La69,Pro59,Ser62,Wah86}.

Local existence and uniqueness of $H^m$ solution can be established
by using analytic semigroup \cite{Lun95} with initial data in
$H^m(\R^3)$, $m\ge1$. This result is stated as follows.

\begin{prop}[Local $H^m$ Solution]\label{NS-LHmS} Let $u_0\in H^m(\R^3)
\cap H(\R^3)$ and $m\ge1$. Then there exist $T_{max}=T_{max}\big(
\|u_0\|_{H^m}\big)>0$ and a unique solution $u$ of the problem
\eqref{NS1}--\eqref{NSi} such that $u\in C\big([0,T_{max});H^m(\R^3)
\cap H(\R^3)\big)$.
\end{prop}

Local existence and uniqueness of mild solution or strong solution
were established \cite{Cal90,GiM85,Ka84,KaF62,KaP94,Wei80} with
initial data in $L^p(\R^3)$, $p>3$. The main result is as follows.

\begin{prop}[Local Mild Solution]\label{NS-LMS} Let $u_0\in L^p(\R^3)$
satisfy \eqref{NS2} in distribution and $p>3$. Then there exist $
T_{max}=T_{max}\big(\|u_0\|_{L^p}\big)>0$ and a unique solution $u$ of
the problem \eqref{NS1}--\eqref{NSi} such that $u\in C([0,T_{max});
L^p(\R^3))$.
\end{prop}

The uniqueness in Proposition \ref{NS-LHmS} and \ref{NS-LMS} ensures
that the symplectic symmetries which are corresponding to initial data
$u_0$ can be kept by the solution $u$.

Besides the local-posedness, the lower bounds of possible blowup
solutions were considered \cite{CSTY08,CSTY09,CMP14,Gig86,Le34,
RSS12}. The concentration phenomena of possible blowup solutions was
studied \cite{LOW18}.

It is well-known that the equations \eqref{NS1} \eqref{NS2} are
scaling-invariant in the sense that if $u$ solves \eqref{NS1}
\eqref{NS2} with initial data $u_0$, so dose $u_{\lam}(t,x)=\lam
u(\lam^2t,\lam x)$ with initial data $\lam u_0(\lam x)$. A space $X$
defined on $\R^3$ is so-called to be critical provided $\|u_0\|_X=
\|\lam u_0(\lam\cdot)\|_X$ for any $\lam>0$. $L^3(\R^3)$ is one of
critical spaces. For the initial data in critical spaces, the
posedness of global solution of the equations \eqref{NS1}--\eqref{NSi}
is obtained \cite{Can97,Che99,KoT01,Pla96} with small
initial data. The regularity criterion was established \cite{ESS03,
GKP13,KeK11,LiW19,Sere12}. On the other hand, the ill-posedness was
showed \cite{BoP08,Ger08,Wang15,Yon10}.

It is also studied that solutions of the problem
\eqref{NS1}--\eqref{NSi} are in various function spaces \cite{Can03,
GiInM99,GiM89,Ka75,KoT01,Tay92}. Partial regularity of suitable weak
solutions was established \cite{CKN82,LaS99,Lin98,Sch76,Sch77,Vas07,
WaW14}. Non-existence of self-similar solutions was proved
\cite{NeRS96,Tsai98}. Decay of the solutions can be found in
\cite{Car96,MiS01,Sch91,Sch92}, etc.

In this paper, we study the global regular solutions of the problem
\eqref{NS1}--\eqref{NSi} with cylindrical symmetry of symplectic
representation.

Now we assume that $\phi$ and $\psi$ are cylindrical symmetric
functions with respect to space variable $x\in\R^3$. It is that
$\phi(t,x)=\phi(t,r,x_3)$, $\psi(t,x)=\psi(t,r,x_3)$  and $r^2=
x_1^2+x_2^2$. Letting $A=e_3=(0,0,1)$ and $\phi=0$, inserting $(\phi,
\psi)=\big(0,\psi(t,r,x_3)\big)$ into equations \eqref{NS-CSR12-phi0}
\eqref{NS-CSR12-psi0}, by long long straightforward calculation,
we derive
\beq\label{CSR12-ps20-ph0}\begin{split}
\patl_r\Del(\psi_t-\nu\Del\psi)
-\Del_r\psi\cdot\patl_r\patl_3\Del\psi
+r\patl_r(\frac1r\patl_r\Del\psi)\cdot\patl_r\patl_3\psi=0,
\end{split}\eeq
where $\Del_r=\patl^2_r+\frac1r\patl_r=r\patl_r(\frac1r\patl_r)
+\frac2r\patl_r$, $\Del=\Del_r+\patl_3^2$.

Suppose that the initial data $u_0$ holds symplectic representation
as follows
\beq\label{IDa-u0ps0}
u_0(x)=\{(A\times\nab)\times\nab\}\psi_0(r,x_3),
\eeq
and $\psi$ satisfies the following initial condition
\beq\label{IDa-ps}
\psi|_{t=0}=\psi_0(r,x_3).
\eeq
Let
\[
u(t,x)=\{(A\times\nab)\times\nab\}\psi(t,r,x_3).
\]
Then the problem \eqref{NS1}--\eqref{NSi} is equivalent to the equation
\eqref{CSR12-ps20-ph0} with initial data \eqref{IDa-ps}.

Firstly we assume that fluid flows in ring cylinder domain
\beq\label{RCD-00}
{\mathbb D}=\{(x_1,x_2,x_3)|0<r_0<r<R_0<\infty,x_3\in\R\},
\eeq
where $r_0$ and $R_0$ are positive constants. Applying equation
\eqref{CSR12-ps20-ph0}, we construct the global solution of problem
\eqref{NS1}--\eqref{NSi} in ${\mathbb D}$ which satisfies a priori
estimate
\beq\label{Est-u-H1D}\begin{split}
&\|u(t,\cdot)\|_{H^1({\mathbb D})}\le C,\;\;\;\;\;\;\forall t\ge0.\\
\end{split}\eeq
Here the constant $C$ is independent of $r_0$ and $R_0$.

Finally let $r_0\rightarrow 0$ and $R_0\rightarrow\infty$, we derive
the main result.

\begin{theorem}[Global Solution]
\label{CSR12-Thm-u} Provided that $m\ge1$, $u_0=\big(\patl_1\patl_3
\psi_0,\patl_2\patl_3\psi_0,-(\patl_1^2+\patl_2^2)\psi_0\big)$,
$\psi_0=\psi_0(r,x_3)$ satisfies
\beq\label{IDa0-upsHm}\begin{split}
&\sum_{j\ge0,\;2j\le m}\int_{\R}\int^{\infty}_0\{(\Del_r\Del^j\psi_0)^2
+(\patl_r\patl_3\Del^j\psi_0)^2\}rdrdx_3\\
&+\sum_{k\ge0,\;2k+1\le m}\int_{\R}\int^{\infty}_0(\patl_r\Del\Del^k
\psi_0)^2rdrdx_3\\
&+\int_{\R}\int^{\infty}_0(\patl_r\Del\psi_0)^2r^{-1}drdx_3<\infty.\\
\end{split}\eeq
Then there exists a unique global solution $u$ of the problem
\eqref{NS1}--\eqref{NSi} such that $u=\big(\patl_1\patl_3\psi,\patl_2
\patl_3\psi,-(\patl_1^2+\patl_2^2)\psi\big)$, $\psi=\psi(t,r,x_3)$ and
\beq\label{GS-uH1}\begin{split}
&u\in C([0,T];H^m(\R^3)\cap H(\R^3))\cap L^2([0,T];H^{m+1}(\R^3)),
\;\;\;\;\;\;\forall T\ge0.\\
\end{split}\eeq
\end{theorem}

The plan of this paper is as follows.
Section 2 is devoted to study that fluid flows in ring cylinder
domain ${\mathbb D}$ and establish a priori uniform estimates.
Section 3 is devoted to investigate that fluid flows in $\R^3$ and
give the proof of Theorem \ref{CSR12-Thm-u}.

\section{Fluid in Ring Cylinder Domain}
\setcounter{equation}{0}

In this section, assume that fluid flows in ring cylinder domain
${\mathbb D}$.

We consider the equation \eqref{CSR12-ps20-ph0} with initial data
\eqref{IDa-ps} and boundary condition
\beq\label{BC-RCD00}\begin{split}
&\patl_r\psi|_{r=r_0}=\patl_r\psi|_{r=R_0}=0,\\
&\patl_r\Del\psi|_{r=r_0}=\patl_r\Del\psi|_{r=R_0}=0.\\
\end{split}\eeq

The equation \eqref{CSR12-ps20-ph0} can be rewritten as
\beq\label{CSR12-ps20-ph0-R}\begin{split}
&\Del(\psi_t-\nu\Del\psi)
-\int^r_{r_0}\Del_s\psi\cdot\patl_s\patl_3\Del\psi(t,s,x_3)ds\\
&+\int^r_{r_0}s\patl_s(\frac1s\patl_s\Del\psi)\cdot\patl_s\patl_3
\psi(t,s,x_3)ds=0.
\end{split}\eeq
Moreover the problem \eqref{CSR12-ps20-ph0-R} \eqref{BC-RCD00} is
equivalent to
\beq\label{Equi-ps-w}\begin{split}
&w_t-\nu\Del w
-\int^r_{r_0}\Del_s\psi\cdot\patl_s\patl_3\Del\psi(t,s,x_3)ds\\
&+\int^r_{r_0}s\patl_s(\frac1s\patl_s\Del\psi)\cdot\patl_s\patl_3
\psi(t,s,x_3)ds=0,\\
&\Del\psi=w,\\
&\frac{\patl}{\patl n}\psi|_{\patl{\mathbb D}}=0,
\;\;\frac{\patl}{\patl n}w|_{\patl{\mathbb D}}=0,\\
\end{split}\eeq
where $n$ is the outer normal of $\patl{\mathbb D}$. Therefore the
existence of local solution of problem \eqref{CSR12-ps20-ph0-R}
\eqref{IDa-ps} \eqref{BC-RCD00} can be established by standard method.

To prove existence of global solution of problem \eqref{CSR12-ps20-ph0}
\eqref{IDa-ps} \eqref{BC-RCD00}, it is sufficient to give some a priori
estimates.

\begin{lemma}\label{Est-L2-RCD} For the solutions of problem
\eqref{CSR12-ps20-ph0} \eqref{IDa-ps} \eqref{BC-RCD00}, we have
\beq\label{Est-L2-RCD00}\begin{split}
&\frac{d}{dt}\int_{\R}\int^{R_0}_{r_0}\{(\Del_r\psi)^2
+(\patl_r\patl_3\psi)^2\}rdrdx_3
+2\nu\int_{\R}\int^{R_0}_{r_0}(\patl_r\Del\psi)^2rdrdx_3=0,\\
\end{split}\eeq
where $\Del_r=\patl_r^2+\frac1r\patl_r$.
\end{lemma}

\begin{proof} Taking inner product of equation \eqref{CSR12-ps20-ph0}
with $r\patl_r\psi$, we have
\beq\label{Est-L2-1}\begin{split}
&\int_{\R}\int^{R_0}_{r_0}\patl_r\Del(\psi_t-\nu\Del\psi)
\patl_r\psi rdrdx_3\\
&-\int_{\R}\int^{R_0}_{r_0}\Del_r\psi\patl_r\patl_3\Del\psi
\patl_r\psi rdrdx_3\\
&+\int_{\R}\int^{R_0}_{r_0}r^2\patl_r(\frac1r\patl_r\Del\psi)
\patl_r\patl_3\psi\patl_r \psi drdx_3=0.
\end{split}\eeq
By part integration, we derive
\[\begin{split}
&\int_{\R}\int^{R_0}_{r_0}\patl_r\Del\psi_t\patl_r\psi rdrdx_3\\
=&-\frac12\frac{d}{dt}\int_{\R}\int^{R_0}_{r_0}\{(\Del_r\psi)^2
+(\patl_r\patl_3\psi)^2\}rdrdx_3,\\
\end{split}\]

\[\begin{split}
&-\nu\int_{\R}\int^{R_0}_{r_0}\patl_r\Del^2\psi\patl_r\psi rdrdx_3\\
=&-\nu\int_{\R}\int^{R_0}_{r_0}(\patl_r\Del_r\psi)^2rdrdx_3\\
&-2\nu\int_{\R}\int^{R_0}_{r_0}(\Del_r\patl_3\psi)^2rdrdx_3\\
&-\nu\int_{\R}\int^{R_0}_{r_0}(\patl_r\patl_3^2\psi)^2rdrdx_3\\
=&-\nu\int_{\R}\int^{R_0}_{r_0}(\patl_r\Del\psi)^2rdrdx_3,\\
\end{split}\]

\[\begin{split}
&-\int_{\R}\int^{R_0}_{r_0}\Del_r\psi\patl_r\patl_3\Del\psi
\patl_r\psi rdrdx_3\\
&+\int_{\R}\int^{R_0}_{r_0}r^2\patl_r(\frac1r\patl_r\Del\psi)
\patl_r\patl_3\psi\patl_r \psi drdx_3\\
=&0.
\end{split}\]
Therefore \eqref{Est-L2-RCD00} is proved.
\end{proof}

\begin{lemma}\label{Est-H1w-RCD} For the solutions of problem
\eqref{CSR12-ps20-ph0} \eqref{IDa-ps} \eqref{BC-RCD00}, we have
\beq\label{Est-H1w-RCD00}\begin{split}
&\frac12\frac{d}{dt}\int_{\R}\int^{R_0}_{r_0}
(\patl_r\Del\psi)^2r^{-1}drdx_3\\
&+\nu\int_{\R}\int^{R_0}_{r_0}\{(\patl^2_r-\frac1r\patl_r)
\Del\psi\}^2r^{-1}drdx_3\\
&+\nu\int_{\R}\int^{R_0}_{r_0}(\patl_r\patl_3\Del\psi)^2r^{-1}drdx_3=0.\\
\end{split}\eeq
\end{lemma}

\begin{proof} Taking inner product of equation \eqref{CSR12-ps20-ph0}
with $r^{-1}\patl_r\Del\psi$, we have
\beq\label{Est-H1w0}\begin{split}
&\int_{\R}\int^{R_0}_{r_0}\patl_r\Del(\psi_t-\nu\Del\psi)
\patl_r\Del\psi r^{-1}drdx_3\\
&-\int_{\R}\int^{R_0}_{r_0}\Del_r\psi\patl_r\patl_3\Del\psi
\patl_r\Del\psi r^{-1}drdx_3\\
&+\int_{\R}\int^{R_0}_{r_0}\patl_r(\frac1r\patl_r\Del\psi)
\patl_r\patl_3\psi\patl_r\Del\psi drdx_3=0.
\end{split}\eeq
By part integration, we derive
\[\begin{split}
&\int_{\R}\int^{R_0}_{r_0}\patl_r\Del\psi_t\patl_r\Del\psi r^{-1}drdx_3
=\frac12\frac{d}{dt}\int_{\R}\int^{R_0}_{r_0}
(\patl_r\Del\psi)^2r^{-1}drdx_3,\\
\end{split}\]

\[\begin{split}
&-\nu\int_{\R}\int^{R_0}_{r_0}\patl_r\Del^2\psi\patl_r\Del\psi r^{-1}drdx_3\\
=&\nu\int_{\R}\int^{R_0}_{r_0}\{\patl^2_r\Del\psi\}^2r^{-1}drdx_3\\
&-\nu\int_{\R}\int^{R_0}_{r_0}\{\frac1r\patl_r\Del\psi\}^2r^{-1}drdx_3\\
&+\nu\int_{\R}\int^{R_0}_{r_0}(\patl_r\patl_3\Del\psi)^2r^{-1}drdx_3,\\
\end{split}\]

\[\begin{split}
&-\int_{\R}\int^{R_0}_{r_0}\Del_r\psi\patl_r\patl_3\Del\psi
\patl_r\Del\psi r^{-1}drdx_3\\
&+\int_{\R}\int^{R_0}_{r_0}\patl_r(\frac1r\patl_r\Del\psi)
\patl_r\patl_3\psi\patl_r\Del\psi drdx_3\\
=&0.\\
\end{split}\]

Therefore we have
\beq\label{Est-H1w1}\begin{split}
&\frac12\frac{d}{dt}\int_{\R}\int^{R_0}_{r_0}
(\patl_r\Del\psi)^2r^{-1}drdx_3\\
&+\nu\int_{\R}\int^{R_0}_{r_0}\{\patl^2_r\Del\psi\}^2r^{-1}drdx_3\\
&-\nu\int_{\R}\int^{R_0}_{r_0}\{\frac1r\patl_r\Del\psi\}^2r^{-1}drdx_3\\
&+\nu\int_{\R}\int^{R_0}_{r_0}(\patl_r\patl_3\Del\psi)^2r^{-1}drdx_3\\
=&0.\\
\end{split}\eeq

Note that
\beq\label{Note-1}\begin{split}
&\int_{\R}\int^{R_0}_{r_0}2\frac1r\patl_r\Del\psi \patl^2_r\Del\psi
r^{-1}drdx_3\\
=&2\int_{\R}\int^{R_0}_{r_0}(\frac1r\patl_r\Del\psi)^2r^{-1}drdx_3.\\
\end{split}\eeq

By \eqref{Est-H1w1} and \eqref{Note-1}, \eqref{Est-H1w-RCD00} is proved.
\end{proof}


\begin{lemma}\label{Est-H1-RCD} For the solutions of problem
\eqref{CSR12-ps20-ph0} \eqref{IDa-ps} \eqref{BC-RCD00}, we have
\beq\label{Est-H1-RCD00}\begin{split}
&\int_{\R}\int^{R_0}_{r_0}(\patl_r\Del\psi)^2rdrdx_3\\
&+\int^t_0\int_{\R}\int^{R_0}_{r_0}\{(\Del_r\Del\psi)^2
+(\patl_r\patl_3\Del\psi)^2\}(s,r,x_3)rdrdx_3ds\\
\le&C\exp\big\{t2\int_{\R}\int^{R_0}_{r_0}(\patl_r\Del\psi_0)^2r^{-1}
drdx_3\big\},\;\;\;\;\;\;\forall t\ge0,\\
\end{split}\eeq
where positive constant $C$ only depends on $\nu$, $\int_{\R}\int^{R_0}_{r_0}
(\patl_r\Del\psi_0)^2rdrdx_3$ and
\[\int_{\R}\int^{R_0}_{r_0}\{(\Del_r\psi_0)^2
+(\patl_r\patl_3\psi_0)^2\}rdrdx_3.\]
But $C$ is independent of $r_0$ and $R_0$.
\end{lemma}

\begin{proof} Taking inner product of equation \eqref{CSR12-ps20-ph0}
with $r\patl_r\Del\psi$, we have
\beq\label{Est-H10}\begin{split}
&\int_{\R}\int^{R_0}_{r_0}\patl_r\Del(\psi_t-\nu\Del\psi)
\patl_r\Del\psi rdrdx_3\\
&-\int_{\R}\int^{R_0}_{r_0}\Del_r\psi\patl_r\patl_3\Del\psi
\patl_r\Del\psi rdrdx_3\\
&+\int_{\R}\int^{R_0}_{r_0}r^2\patl_r(\frac1r\patl_r\Del\psi)
\patl_r\patl_3\psi\patl_r\Del\psi drdx_3=0.
\end{split}\eeq

By part integration, we derive
\[\begin{split}
&\int_{\R}\int^{R_0}_{r_0}\patl_r\Del\psi_t\patl_r\Del\psi rdrdx_3
=\frac12\frac{d}{dt}\int_{\R}\int^{R_0}_{r_0}
(\patl_r\Del\psi)^2rdrdx_3,\\
\end{split}\]

\[\begin{split}
&-\nu\int_{\R}\int^{R_0}_{r_0}\patl_r\Del^2\psi\patl_r\Del\psi rdrdx_3\\
=&\nu\int_{\R}\int^{R_0}_{r_0}(\Del_r\Del\psi)^2rdrdx_3
+\nu\int_{\R}\int^{R_0}_{r_0}(\patl_r\patl_3\Del\psi)^2rdrdx_3,\\
\end{split}\]

\[\begin{split}
&-\int_{\R}\int^{R_0}_{r_0}\Del_r\psi\patl_r\patl_3\Del\psi
\patl_r\Del\psi rdrdx_3\\
&+\int_{\R}\int^{R_0}_{r_0}r^2\patl_r(\frac1r\patl_r\Del\psi)
\patl_r\patl_3\psi\patl_r\Del\psi drdx_3\\
=&-\int_{\R}\int^{R_0}_{r_0}(\patl_r\Del\psi)^2
\patl_r\patl_3\psi drdx_3.\\
\end{split}\]

Therefore we have
\beq\label{Est-H1-1}\begin{split}
&\frac12\frac{d}{dt}\int_{\R}\int^{R_0}_{r_0}
(\patl_r\Del\psi)^2rdrdx_3\\
&+\nu\int_{\R}\int^{R_0}_{r_0}\{(\Del_r\Del\psi)^2
+(\patl_r\patl_3\Del\psi)^2\}rdrdx_3\\
&-\int_{\R}\int^{R_0}_{r_0}(\patl_r\Del\psi)^2
\patl_r\patl_3\psi drdx_3=0.\\
\end{split}\eeq

By H$\ddot{o}$lder's inequality, we obtain
\beq\label{Est-HI-0}\begin{split}
&\Big|\int_{\R}\int^{R_0}_{r_0}(\patl_r\Del\psi)^2
\patl_r\patl_3\psi(t,r,x_3)drdx_3\Big|\\
\le&\Big(\int(\patl_r\Del\psi_0)^2r^{-1}drdx_3\Big)^{1/2}
\Big(\int(\patl_r\Del\psi)^2rdrdx_3\Big)^{1/2}
\|\patl_r\patl_3\psi(t,\cdot)\|_{L^{\infty}}\\
\le&\|\patl_r\patl_3\psi(t,\cdot)\|_{L^{\infty}}^2
+\int(\patl_r\Del\psi_0)^2r^{-1}drdx_3\int(\patl_r\Del\psi)^2rdrdx_3,\\
\end{split}\eeq
where the estimate \eqref{Est-H1w-RCD00} is used.

We can establish the following estimate
\beq\label{Est-LI-rz0}\begin{split}
&(\patl_r\patl_3\psi)^2(t,r,x_3)=2\int^r_{r_0}\patl_s^2\patl_3\psi
\patl_s\patl_3\psi(t,s,x_3)ds\\
\le&4\int^r_{r_0}\big\{\|\patl_s^2\patl_3^2\psi(t,s,\cdot)\|_{L^2}^{1/2}
\|\patl_s^2\patl_3\psi(t,s,\cdot)\|_{L^2}^{1/2}\\
&\hspace{12mm}\|\patl_s\patl_3^2\psi(t,s,\cdot)\|_{L^2}^{1/2}
\|\patl_s\patl_3\psi(t,s,\cdot)\|_{L^2}^{1/2}\big\}ds,\\
\end{split}\eeq
where we have used the following inequality
\beq\label{Est-LI-1d}
\|f\|_{L^{\infty}(\R)}\le\sqrt{2}\|\patl f\|_{L^2(\R)}^{1/2}
\|f\|_{L^2(\R)}^{1/2}.
\eeq
Thus we derive
\beq\label{Est-LI-rz1}\begin{split}
&\|\patl_r\patl_3\psi(t,\cdot)\|_{L^{\infty}}^2
=\|(\patl_r\patl_3\psi)^2(t,\cdot)\|_{L^{\infty}}\\
\le&4\int^{R_0}_{r_0}\big\{\|\patl_r^2\patl_3^2\psi(t,r,\cdot)\|_{L^2}^{1/2}
\|\patl_r^2\patl_3\psi(t,r,\cdot)\|_{L^2}^{1/2}\\
&\hspace{13mm}\|\patl_r\patl_3^2\psi(t,r,\cdot)\|_{L^2}^{1/2}
\|\patl_r\patl_3\psi(t,r,\cdot)\|_{L^2}^{1/2}\big\}dr\\
\le&\frac{\nu}4\int_{\R}\int^{R_0}_{r_0}\{\big(\patl_r^2\patl_3^2\psi\big)^2r
+\big(\patl_r\patl_3^2\psi\big)^2r^{-1}\}drdx_3\\
&+\frac4{\nu}\int_{\R}\int^{R_0}_{r_0}\{\big(\patl_r^2\patl_3\psi\big)^2r
+\big(\patl_r\patl_3\psi\big)^2r^{-1}\}drdx_3.\\
\end{split}\eeq

There exist the following observations
\beq\label{Eqv-Obs-H3}\begin{split}
&\int_{\R}\int^{R_0}_{r_0}(\patl_r\Del\psi)^2rdrdx_3\\
=&\int_{\R}\int^{R_0}_{r_0}(\patl_r\Del_r\psi)^2rdrdx_3
+\int_{\R}\int^{R_0}_{r_0}(\patl_r\patl_3^2\psi)^2rdrdx_3\\
&+2\int_{\R}\int^{R_0}_{r_0}(\Del_r\patl_3\psi)^2rdrdx_3\\
=&\int_{\R}\int^{R_0}_{r_0}(\patl_r\Del_r\psi)^2rdrdx_3
+\int_{\R}\int^{R_0}_{r_0}(\patl_r\patl_3^2\psi)^2rdrdx_3\\
&+2\int_{\R}\int^{R_0}_{r_0}(\patl_r^2\patl_3\psi)^2rdrdx_3
+2\int_{\R}\int^{R_0}_{r_0}(\patl_r\patl_3\psi)^2r^{-1}drdx_3,\\
\end{split}\eeq

\beq\label{Eqv-Obs-H4}\begin{split}
&\int_{\R}\int^{R_0}_{r_0}(\Del_r\Del\psi)^2rdrdx_3\\
=&\int_{\R}\int^{R_0}_{r_0}(\Del_r^2\psi)^2rdrdx_3\\
&+\int_{\R}\int^{R_0}_{r_0}\{(\patl^2_r\patl_3^2\psi)^2r
+(\patl_r\patl_3^2\psi)^2r^{-1}\}drdx_3\\
&+2\int_{\R}\int^{R_0}_{r_0}(\patl_r\Del_r\patl_3\psi)^2rdrdx_3.\\
\end{split}\eeq

Applying \eqref{Est-L2-RCD00} and \eqref{Eqv-Obs-H3}, we get
\beq\label{Est-IH1-1}\begin{split}
&\int^t_0\int_{\R}\int^{R_0}_{r_0}\{(\patl^2_r\patl_3\psi)^2r
+(\patl_r\patl_3\psi)^2r^{-1}\}(s,r,x_3)drdx_3ds\\
\le&\frac1{4\nu}\int_{\R}\int^{R_0}_{r_0}\{(\Del_r\psi_0)^2
+(\patl_r\patl_3\psi_0)^2\}rdrdx_3,
\;\;\;\;\;\;\forall t\ge0.\\
\end{split}\eeq

Inserting \eqref{Est-HI-0} \eqref{Est-LI-rz1} \eqref{Eqv-Obs-H4} into
\eqref{Est-H1-1}, we have
\beq\label{Est-H1-2}\begin{split}
&\frac12\frac{d}{dt}\int_{\R}\int^{R_0}_{r_0}
(\patl_r\Del\psi)^2rdrdx_3\\
&+\frac{\nu}2\int_{\R}\int^{R_0}_{r_0}(\Del_r\Del\psi)^2rdrdx_3
+\nu\int_{\R}\int^{R_0}_{r_0}(\patl_r\patl_3\Del\psi)^2rdrdx_3\\
\le&\int_{\R}\int^{R_0}_{r_0}(\patl_r\Del\psi_0)^2r^{-1}drdx_3
\int_{\R}\int^{R_0}_{r_0}(\patl_r\Del\psi)^2rdrdx_3\\
&+\frac4{\nu}\int_{\R}\int^{R_0}_{r_0}\{\big(\patl_r^2\patl_3\psi\big)^2r
+\big(\patl_r\patl_3\psi(t,r,x_3)\big)^2r^{-1}\}drdx_3.\\
\end{split}\eeq

Employing \eqref{Est-IH1-1} \eqref{Est-H1-2} and Gronwall's inequality,
we derive \eqref{Est-H1-RCD00}.
\end{proof}

\begin{lemma}\label{Est-H2-RCD} For the solutions of problem
\eqref{CSR12-ps20-ph0} \eqref{IDa-ps} \eqref{BC-RCD00}, we have
\beq\label{Est-H2-RCD00}\begin{split}
&\int_{\R}\int^{R_0}_{r_0}\{(\Del_r\Del\psi)^2
+(\patl_r\patl_3\Del\psi)^2\}rdrdx_3\\
&+\int^t_0\int_{\R}\int^{R_0}_{r_0}\{\patl_r\Del\psi_t\}^2
(s,r,x_3)rdrdx_3ds\\
&+\int^t_0\int_{\R}\int^{R_0}_{r_0}(\patl_r\Del^2\psi)^2
(s,r,x_3)rdrdx_3ds\\
\le&C_2,\;\;\;\;\;\;\forall t\ge0,\\
\end{split}\eeq
where positive constant $C_2$ depends on $t$, $\nu$, $r_0$, $R_0$ and
\[\begin{split}
&\int_{\R}\int^{R_0}_{r_0}(\patl_r\Del\psi_0)^2rdrdx_3,\;\;
\int_{\R}\int^{R_0}_{r_0}(\patl_r\Del\psi_0)^2r^{-1}drdx_3,\\
&\int_{\R}\int^{R_0}_{r_0}(\Del_r\Del^k\psi_0)^2rdrdx_3
+\int_{\R}\int^{R_0}_{r_0}(\patl_r\patl_3\Del^k\psi_0)^2rdrdx_3\;\;(k=0,1).\\
\end{split}\]
\end{lemma}

\begin{proof} Taking inner product of equation \eqref{CSR12-ps20-ph0}
with $r\patl_r\Del\psi_t$, we have
\beq\label{Est-H20}\begin{split}
&\int_{\R}\int^{R_0}_{r_0}\patl_r\Del(\psi_t-\nu\Del\psi)
\patl_r\Del\psi_trdrdx_3\\
&-\int_{\R}\int^{R_0}_{r_0}\Del_r\psi\patl_r\patl_3\Del\psi
\patl_r\Del\psi_trdrdx_3\\
&+\int_{\R}\int^{R_0}_{r_0}\patl_r(\frac1r\patl_r\Del\psi)
\patl_r\patl_3\psi\patl_r\Del\psi_t r^2drdx_3=0.
\end{split}\eeq

Employing equation \eqref{CSR12-ps20-ph0}, we obtain
\beq\label{Est-H20-t2r}\begin{split}
&\nu^2\int_{\R}\int^{R_0}_{r_0}(\patl_r\Del^2\psi)^2rdrdx_3\\
\le&3\int_{\R}\int^{R_0}_{r_0}(\patl_r\Del\psi_t)^2rdrdx_3\\
&+3\int_{\R}\int^{R_0}_{r_0}(\Del_r\psi\patl_r\patl_3\Del\psi)^2rdrdx_3\\
&+3\int_{\R}\int^{R_0}_{r_0}\{\patl_r(\frac1r\patl_r\Del\psi)
\patl_r\patl_3\psi\}^2 r^3drdx_3.
\end{split}\eeq

By part integration, we derive
\beq\label{Est-4r-dt}\begin{split}
&-\nu\int_{\R}\int^{R_0}_{r_0}\patl_r\Del^2\psi\patl_r\Del\psi_t rdrdx_3\\
=&\frac{\nu}2\frac{d}{dt}\int_{\R}\int^{R_0}_{r_0}
(\Del_r\Del\psi)^2rdrdx_3\\
&+\frac{\nu}2\frac{d}{dt}\int_{\R}\int^{R_0}_{r_0}
(\patl_r\patl_3\Del\psi)^2rdrdx_3.\\
\end{split}\eeq

By H$\ddot{o}$lder's inequality, we obtain
\beq\label{Est-H20-NL10}\begin{split}
&\Big|\int_{\R}\int^{R_0}_{r_0}\Del_r\psi\patl_r\patl_3\Del\psi
\patl_r\Del\psi_trdrdx_3\Big|\\
\le&\frac14\int_{\R}\int^{R_0}_{r_0}(\patl_r\Del\psi_t)^2rdrdx_3
+16\int_{\R}\int^{R_0}_{r_0}(\Del_r\psi\patl_r\patl_3\Del\psi)^2rdrdx_3,\\
\end{split}\eeq

\beq\label{Est-H20-NL20}\begin{split}
&\Big|\int_{\R}\int^{R_0}_{r_0}\patl_r(\frac1r\patl_r\Del\psi)
\patl_r\patl_3\psi\patl_r\Del\psi_t r^2drdx_3\Big|\\
\le&\frac14\int_{\R}\int^{R_0}_{r_0}(\patl_r\Del\psi_t)^2rdrdx_3
+16\int_{\R}\int^{R_0}_{r_0}\{\patl_r(\frac1r\patl_r\Del\psi)
\patl_r\patl_3\psi\}^2r^3drdx_3.\\
\end{split}\eeq

Inserting \eqref{Est-H20-t2r}--\eqref{Est-H20-NL20} into
\eqref{Est-H20}, we derive
\beq\label{Est-H201}\begin{split}
&\nu\frac{d}{dt}\int_{\R}\int^{R_0}_{r_0}(\Del_r\Del\psi)^2rdrdx_3
+\nu\frac{d}{dt}\int_{\R}\int^{R_0}_{r_0}(\patl_r\patl_3\Del\psi)^2rdrdx_3\\
&+\frac{\nu^2}6\int_{\R}\int^{R_0}_{r_0}(\patl_r\Del^2\psi)^2rdrdx_3
+\frac12\int_{\R}\int^{R_0}_{r_0}(\patl_r\Del\psi_t)^2rdrdx_3\\
\le&33\int_{\R}\int^{R_0}_{r_0}(\Del_r\psi\patl_r\patl_3\Del\psi)^2rdrdx_3
+33\int_{\R}\int^{R_0}_{r_0}\{\patl_r(\frac1r\patl_r\Del\psi)
\patl_r\patl_3\psi\}^2 r^3drdx_3.
\end{split}\eeq

It is obvious that
\beq\label{Est-H20-NL11}\begin{split}
&\int_{\R}\int^{R_0}_{r_0}(\Del_r\psi\patl_r\patl_3\Del\psi)^2rdrdx_3\\
\le&\|\Del_r\psi(t,\cdot)\|_{L^{\infty}}^2
\int_{\R}\int^{R_0}_{r_0}(\patl_r\patl_3\Del\psi)^2rdrdx_3,\\
\end{split}\eeq

\beq\label{Est-H20-NL21}\begin{split}
&\int_{\R}\int^{R_0}_{r_0}\{\patl_r(\frac1r\patl_r\Del\psi)
\patl_r\patl_3\psi\}^2r^3drdx_3\\
\le&2\int_{\R}\int^{R_0}_{r_0}\{(\patl_r^2\Del\psi\patl_r\patl_3\psi)^2r
+(\patl_r\Del\psi\patl_r\patl_3\psi)^2r^{-1}\}drdx_3\\
\le&2\|\patl_r\patl_3\psi(t,\cdot)\|_{L^{\infty}}^2\Big\{
\int_{\R}\int^{R_0}_{r_0}(\patl_r^2\Del\psi)^2rdrdx_3\\
&+\int_{\R}\int^{R_0}_{r_0}(\patl_r\Del\psi)^2r^{-1}drdx_3\Big\}.\\
\end{split}\eeq

By Gagliardo-Nirenberg's inequality, we get
\beq\label{Est-LI-rr0}\begin{split}
&\|\Del_r\psi(t,\cdot)\|_{L^{\infty}({\mathbb D})}
\le C\|\Del_r\psi(t,\cdot)\|_{H^2({\mathbb D})}^{1/2}
\|\Del_r\psi(t,\cdot)\|_{L^2({\mathbb D})}^{1/2}\\
\le&C\Big\{\int_{\R}\int^{R_0}_{r_0}\{(\patl_r^2\Del_r\psi)^2
+(\patl_r\patl_3\Del_r\psi)^2+(\patl_3^2\Del_r\psi)^2\}rdrdx_3\\
&+\int_{\R}\int^{R_0}_{r_0}(\Del_r\psi)^2rdrdx_3\Big\}^{1/4}
\Big\{\int_{\R}\int^{R_0}_{r_0}(\Del_r\psi)^2rdrdx_3\Big\}^{1/4}\\
\le&C\Big\{\int_{\R}\int^{R_0}_{r_0}\{(\patl_r^2\Del_r\psi)^2
+(\patl_r\patl_3\Del_r\psi)^2+(\patl_3^2\Del_r\psi)^2\}rdrdx_3
+1\Big\}^{1/4},\\
\end{split}\eeq
where we have used the estimation \eqref{Est-L2-RCD00}, $L^2$ norm is
defined by
\[
\|f\|_{L^2({\mathbb D})}^2=\int_{\R}\int^{R_0}_{r_0}f^2(r,x_3)drdx_3.
\]

There exist the following observations
\beq\label{Eqv-Obs-H4-r}\begin{split}
&\int_{\R}\int^{R_0}_{r_0}(\Del_r^2\psi)^2rdrdx_3\\
=&\int_{\R}\int^{R_0}_{r_0}\{(\patl_r^2\Del_r\psi)^2r
+(\patl_r\Del_r\psi)^2r^{-1}\}drdx_3,\\
\end{split}\eeq

\beq\label{Eqv-Obs-H4-z}\begin{split}
&\int_{\R}\int^{R_0}_{r_0}(\patl_3^2\Del_r\psi)^2rdrdx_3\\
=&\int_{\R}\int^{R_0}_{r_0}\{(\patl_r^2\patl_3^2\psi)^2r
+(\patl_r\patl_3^2\psi)^2r^{-1}\}drdx_3,\\
\end{split}\eeq

\beq\label{Eqv-Obs-H4r2}\begin{split}
&\int_{\R}\int^{R_0}_{r_0}(\patl_r^2\Del\psi)^2rdrdx_3\\
=&\int_{\R}\int^{R_0}_{r_0}\{(\Del_r-\frac1r\patl_r)\Del\psi\}^2rdrdx_3\\
\le&2\int_{\R}\int^{R_0}_{r_0}(\patl_r\Del\psi)^2r^{-1}drdx_3
+2\int_{\R}\int^{R_0}_{r_0}(\Del_r\Del\psi)^2rdrdx_3.\\
\end{split}\eeq

Inserting \eqref{Est-LI-rr0} \eqref{Eqv-Obs-H4} \eqref{Eqv-Obs-H4-r}
\eqref{Eqv-Obs-H4-z} into \eqref{Est-H20-NL11}, we have
\beq\label{Est-H20-NL12}\begin{split}
&33\int_{\R}\int^{R_0}_{r_0}(\Del_r\psi\patl_r\patl_3\Del\psi)^2rdrdx_3\\
\le&C\int_{\R}\int^{R_0}_{r_0}(\patl_r\patl_3\Del\psi)^2rdrdx_3
\Big\{\int_{\R}\int^{R_0}_{r_0}(\Del_r\Del\psi)^2rdrdx_3+1\Big\}^{1/2},\\
\le&C\Big\{\int_{\R}\int^{R_0}_{r_0}(\patl_r\patl_3\Del\psi)^2rdrdx_3\Big\}^2
+C\int_{\R}\int^{R_0}_{r_0}(\Del_r\Del\psi)^2rdrdx_3+C.\\
\end{split}\eeq

Inserting \eqref{Est-LI-rz1} \eqref{Eqv-Obs-H4} \eqref{Eqv-Obs-H3}
\eqref{Eqv-Obs-H4r2} into \eqref{Est-H20-NL21}, we have
\beq\label{Est-H20-NL22}\begin{split}
&33\int_{\R}\int^{R_0}_{r_0}\{\patl_r(\frac1r\patl_r\Del\psi)
\patl_r\patl_3\psi\}^2r^3drdx_3\\
\le C&\Big\{\int_{\R}\int^{R_0}_{r_0}(\Del_r\Del\psi)^2rdrdx_3
+\int_{\R}\int^{R_0}_{r_0}(\patl_r\Del\psi)^2rdrdx_3\Big\}\\
&\cdot\Big\{\int_{\R}\int^{R_0}_{r_0}(\Del_r\Del\psi)^2rdrdx_3
+\int_{\R}\int^{R_0}_{r_0}(\patl_r\Del\psi)^2r^{-1}drdx_3\Big\}\\
\le C&\Big\{\int_{\R}\int^{R_0}_{r_0}(\Del_r\Del\psi)^2rdrdx_3\Big\}^2+C,\\
\end{split}\eeq
where we have used the estimations \eqref{Est-H1-RCD00} and
\eqref{Est-H1w-RCD00}.

Inserting \eqref{Est-H20-NL12} and \eqref{Est-H20-NL22} into
\eqref{Est-H201}, applying Gronwall's inequality and the estimation
\eqref{Est-H1-RCD00}, we derive \eqref{Est-H2-RCD00}.
\end{proof}

\begin{lemma}\label{Est-r2t1-RCD} For the solutions of problem
\eqref{CSR12-ps20-ph0} \eqref{IDa-ps} \eqref{BC-RCD00}, we have
\beq\label{Est-r2t1-RCD00}\begin{split}
\int^t_0\int_{\R}\int^{R_0}_{r_0}\{(\Del_r\psi_t)^2
+(\patl_r\patl_3\psi_t)^2\}(s,r,x_3)rdrdx_3ds\le C_2,
\;\;\;\;\;\;\forall t\ge0,\\
\end{split}\eeq
where positive constant $C_2$ depends on $t$, $\nu$, $r_0$, $R_0$ and
\[\begin{split}
&\int_{\R}\int^{R_0}_{r_0}(\patl_r\Del\psi_0)^2rdrdx_3,\;\;
\int_{\R}\int^{R_0}_{r_0}(\patl_r\Del\psi_0)^2r^{-1}drdx_3,\\
&\int_{\R}\int^{R_0}_{r_0}(\Del_r\Del^k\psi_0)^2rdrdx_3
+\int_{\R}\int^{R_0}_{r_0}(\patl_r\patl_3\Del^k\psi_0)^2rdrdx_3\;\;(k=0,1).\\
\end{split}\]
\end{lemma}

\begin{proof} Taking inner product of the equation
\eqref{CSR12-ps20-ph0} with $r\patl_r\psi_t$, we have
\beq\label{Est-r2t10}\begin{split}
&\int_{\R}\int^{R_0}_{r_0}\patl_r\Del(\psi_t-\nu\Del\psi)
\patl_r\psi_trdrdx_3\\
&-\int_{\R}\int^{R_0}_{r_0}\Del_r\psi\patl_r\patl_3\Del\psi
\patl_r\psi_trdrdx_3\\
&+\int_{\R}\int^{R_0}_{r_0}\patl_r(\frac1r\patl_r\Del\psi)
\patl_r\patl_3\psi\patl_r\psi_tr^2drdx_3=0.
\end{split}\eeq

By part integration, we derive
\beq\label{Dr2t1}\begin{split}
&\int_{\R}\int^{R_0}_{r_0}\patl_r\Del\psi_t\patl_r\psi_trdrdx_3
=-\int_{\R}\int^{R_0}_{r_0}\{(\Del_r\psi_t)^2
+(\patl_r\patl_3\psi_t)^2rdrdx_3,\\
\end{split}\eeq

\beq\label{Dtr3}\begin{split}
&\nu\int_{\R}\int^{R_0}_{r_0}\patl_r\Del^2\psi\patl_r\psi_trdrdx_3\\
=&\frac{\nu}2\frac d{dt}\int_{\R}\int^{R_0}_{r_0}\{(\patl_r\Del_r\psi)^2
+2(\Del_r\patl_3\psi)^2+(\patl_r\patl_3^2\psi)^2\}rdrdx_3,\\
\end{split}\eeq

\beq\label{NT-r2t1}\begin{split}
&-\int_{\R}\int^{R_0}_{r_0}\Del_r\psi\patl_r\patl_3\Del\psi
\patl_r\psi_trdrdx_3\\
&+\int_{\R}\int^{R_0}_{r_0}\patl_r(\frac1r\patl_r\Del\psi)
\patl_r\patl_3\psi\patl_r\psi_tr^2drdx_3\\
=&\int_{\R}\int^{R_0}_{r_0}\Del_r\psi\patl_r\Del\psi
\patl_r\patl_3\psi_trdrdx_3\\
&-\int_{\R}\int^{R_0}_{r_0}\patl_r\Del\psi
\patl_r\patl_3\psi\Del_r\psi_trdrdx_3.\\
\end{split}\eeq

By H$\ddot{o}$lder's inequality, we obtain
\beq\label{NT-r2t10}\begin{split}
&\Big|\int_{\R}\int^{R_0}_{r_0}\Del_r\psi\patl_r\patl_3\Del\psi
\patl_r\psi_trdrdx_3
-\int_{\R}\int^{R_0}_{r_0}\patl_r(\frac1r\patl_r\Del\psi)
\patl_r\patl_3\psi\patl_r\psi_tr^2drdx_3\Big|\\
\le&\Big|\int_{\R}\int^{R_0}_{r_0}\Del_r\psi\patl_r\Del\psi
\patl_r\patl_3\psi_trdrdx_3\Big|
+\Big|\int_{\R}\int^{R_0}_{r_0}\patl_r\Del\psi
\patl_r\patl_3\psi\Del_r\psi_trdrdx_3\Big|\\
\le&\|\Del_r\psi(t,\cdot)\|_{L^{\infty}}\Big\{\int_{\R}\int^{R_0}_{r_0}
(\patl_r\Del\psi)^2rdrdx_3\Big\}^{1/2}\Big\{\int_{\R}\int^{R_0}_{r_0}
(\patl_r\patl_3\psi_t)^2rdrdx_3\Big\}^{1/2}\\
&+\|\patl_r\patl_3\psi(t,\cdot)\|_{L^{\infty}}\Big\{\int_{\R}
\int^{R_0}_{r_0}(\patl_r\Del\psi)^2rdrdx_3\Big\}^{1/2}\Big\{\int_{\R}
\int^{R_0}_{r_0}(\Del_r\psi_t)^2rdrdx_3\Big\}^{1/2}\\
\le&C\|\Del_r\psi(t,\cdot)\|_{L^{\infty}}\Big\{\int_{\R}\int^{R_0}_{r_0}
(\patl_r\patl_3\psi_t)^2rdrdx_3\Big\}^{1/2}\\
&+C\|\patl_r\patl_3\psi(t,\cdot)\|_{L^{\infty}}\Big\{\int_{\R}
\int^{R_0}_{r_0}(\Del_r\psi_t)^2rdrdx_3\Big\}^{1/2},\\
\end{split}\eeq
where we have used the estimation \eqref{Est-H1-RCD00}.

Inserting  \eqref{Est-LI-rr0} \eqref{Eqv-Obs-H4-r} \eqref{Eqv-Obs-H4-z}
\eqref{Eqv-Obs-H4} \eqref{Est-LI-rz1} \eqref{Eqv-Obs-H3} into
\eqref{NT-r2t10}, we get
\beq\label{NT-r2t11}\begin{split}
&\Big|\int_{\R}\int^{R_0}_{r_0}\Del_r\psi\patl_r\patl_3\Del\psi
\patl_r\psi_trdrdx_3-\int_{\R}\int^{R_0}_{r_0}\patl_r(\frac1r\patl_r
\Del\psi)\patl_r\patl_3\psi\patl_r\psi_tr^2drdx_3\Big|\\
\le&C\Big\{\int_{\R}\int^{R_0}_{r_0}(\patl_r\patl_3\psi_t)^2rdrdx_3\Big\}^{1/2}
+C\Big\{\int_{\R}\int^{R_0}_{r_0}(\Del_r\psi_t)^2rdrdx_3\Big\}^{1/2}\\
\le&\frac12\int_{\R}\int^{R_0}_{r_0}\{(\Del_r\psi_t)^2
+(\patl_r\patl_3\psi_t)^2\}rdrdx_3+C,\\
\end{split}\eeq
where we have used the estimations \eqref{Est-H1-RCD00} and \eqref{Est-H2-RCD00}.

Inserting \eqref{Dr2t1} \eqref{Dtr3} \eqref{Eqv-Obs-H3} \eqref{NT-r2t11}
into \eqref{Est-r2t10}, integrating with respect to $t$, we derive
\eqref{Est-r2t1-RCD00}.
\end{proof}

\begin{lemma}\label{Est-H3-RCD} For the solutions of problem
\eqref{CSR12-ps20-ph0} \eqref{IDa-ps} \eqref{BC-RCD00}, we have
\beq\label{Est-H3-RCD00}\begin{split}
&\int_{\R}\int^{R_0}_{r_0}\{(\patl_r\Del\psi_t)^2
+(\patl_r\Del^2\psi)^2\}rdrdx_3\\
&+\int^t_0\int_{\R}\int^{R_0}_{r_0}\{\Del_r\Del\psi_t\}^2
(s,r,x_3)rdrdx_3ds\\
&+\int^t_0\int_{\R}\int^{R_0}_{r_0}\{\Del_r\Del^2\psi\}^2
(s,r,x_3)rdrdx_3ds\\
&+\int^t_0\int_{\R}\int^{R_0}_{r_0}(\patl_r\patl_3\Del\psi_t)^2
(s,r,x_3)rdrdx_3ds\\
&+\int^t_0\int_{\R}\int^{R_0}_{r_0}(\patl_r\patl_3\Del^2\psi)^2
(s,r,x_3)rdrdx_3ds\\
\le&C_3,\;\;\;\;\;\;\forall t\ge0,\\
\end{split}\eeq
where positive constant $C_3$ depends on $t$, $\nu$, $r_0$, $R_0$ and
\[\begin{split}
&\int_{\R}\int^{R_0}_{r_0}(\patl_r\Del^m\psi_0)^2rdrdx_3\;\;(m=1,2),\;\;
\int_{\R}\int^{R_0}_{r_0}(\patl_r\Del\psi_0)^2r^{-1}drdx_3,\\
&\int_{\R}\int^{R_0}_{r_0}(\Del_r\Del^k\psi_0)^2rdrdx_3
+\int_{\R}\int^{R_0}_{r_0}(\patl_r\patl_3\Del^k\psi_0)^2rdrdx_3\;\;(k=0,1).\\
\end{split}\]
\end{lemma}

\begin{proof} Differentiating equation \eqref{CSR12-ps20-ph0} with
respect to $t$, we obtian
\beq\label{CSR12-ps20-t2}\begin{split}
\patl_r\Del(\psi_{tt}-\nu\Del\psi_t)-\{\Del_r\psi\patl_r\patl_3\Del
\psi\}_t+\{\patl_r(\frac1r\patl_r\Del\psi)\patl_r\patl_3\psi\}_tr=0.
\end{split}\eeq

Taking inner product of the equation \eqref{CSR12-ps20-t2} with
$r\patl_r\Del\psi_t$, we have
\beq\label{Est-H330}\begin{split}
&\int_{\R}\int^{R_0}_{r_0}\patl_r\Del(\psi_{tt}-\nu\Del\psi_t)
\patl_r\Del\psi_trdrdx_3\\
&-\int_{\R}\int^{R_0}_{r_0}\{\Del_r\psi\patl_r\patl_3\Del\psi\}_t
\patl_r\Del\psi_trdrdx_3\\
&+\int_{\R}\int^{R_0}_{r_0}\{\patl_r(\frac1r\patl_r\Del\psi)
\patl_r\patl_3\psi\}_t\patl_r\Del\psi_tr^2drdx_3=0.
\end{split}\eeq

By part integration, we derive
\beq\label{Dt330}\begin{split}
&\int_{\R}\int^{R_0}_{r_0}\patl_r\Del\psi_{tt}\patl_r\Del\psi_trdrdx_3
=\frac12\frac{d}{dt}\int_{\R}\int^{R_0}_{r_0}
(\patl_r\Del\psi_t)^2rdrdx_3,\\
\end{split}\eeq

\beq\label{Dr330}\begin{split}
&-\nu\int_{\R}\int^{R_0}_{r_0}\patl_r\Del^2\psi_t\patl_r\Del\psi_trdrdx_3\\
=&\nu\int_{\R}\int^{R_0}_{r_0}(\Del_r\Del\psi_t)^2rdrdx_3
+\nu\int_{\R}\int^{R_0}_{r_0}(\patl_r\patl_3\Del\psi_t)^2rdrdx_3,\\
\end{split}\eeq

\beq\label{NL330}\begin{split}
&-\int_{\R}\int^{R_0}_{r_0}\{\Del_r\psi\patl_r\patl_3\Del\psi\}_t
\patl_r\Del\psi_trdrdx_3\\
&+\int_{\R}\int^{R_0}_{r_0}\{\patl_r(\frac1r\patl_r\Del\psi)
\patl_r\patl_3\psi\}_t\patl_r\Del\psi_tr^2drdx_3\\
=&-\int_{\R}\int^{R_0}_{r_0}\frac1r\patl_r\patl_3\psi
(\patl_r\Del\psi_t)^2rdrdx_3\\
&+\int_{\R}\int^{R_0}_{r_0}\Del_r\psi_t\patl_r\Del\psi
\patl_r\patl_3\Del\psi_trdrdx_3
-\int_{\R}\int^{R_0}_{r_0}\patl_r\Del\psi
\patl_r\patl_3\psi_t\Del_r\Del\psi_trdrdx_3.\\
\end{split}\eeq

By H$\ddot{o}$lder's inequality, we obtain
\beq\label{NT330}\begin{split}
&\Big|-\int_{\R}\int^{R_0}_{r_0}\{\Del_r\psi\patl_r\patl_3\Del\psi\}_t
\patl_r\Del\psi_trdrdx_3\\
&+\int_{\R}\int^{R_0}_{r_0}\{\patl_r(\frac1r\patl_r\Del\psi)
\patl_r\patl_3\psi\}_t\patl_r\Del\psi_tr^2drdx_3\Big|\\
\le&r_0^{-1}\|\patl_r\patl_3\psi(t,\cdot)\|_{L^{\infty}}\int_{\R}
\int^{R_0}_{r_0}(\patl_r\Del\psi_t)^2rdrdx_3\\
&+\|\patl_r\Del\psi(t,\cdot)\|_{L^{\infty}}\Big\{\int_{\R}
\int^{R_0}_{r_0}(\Del_r\psi_t)^2rdrdx_3\Big\}^{1/2}\\
&\hspace{5mm}\cdot\Big\{\int_{\R}\int^{R_0}_{r_0}(\patl_r\patl_3\Del
\psi_t)^2rdrdx_3\Big\}^{1/2}\\
&+\|\patl_r\Del\psi(t,\cdot)\|_{L^{\infty}}\Big\{\int_{\R}
\int^{R_0}_{r_0}(\patl_r\patl_3\psi_t)^2rdrdx_3\Big\}^{1/2}\\
&\hspace{5mm}\cdot\Big\{\int_{\R}\int^{R_0}_{r_0}(\Del_r\Del
\psi_t)^2rdrdx_3\Big\}^{1/2}.\\
\end{split}\eeq

Using the same argument as in the proof of \eqref{Est-LI-rz0}
\eqref{Est-LI-rz1}, we get
\beq\label{Est-LI-r30}\begin{split}
&\|\patl_r\Del\psi(t,\cdot)\|_{L^{\infty}}^2
=\|(\patl_r\Del\psi)^2(t,\cdot)\|_{L^{\infty}}\\
\le&4\int^{R_0}_{r_0}\big\{\|\patl_r^2\Del\patl_3\psi(t,r,\cdot)
\|_{L^2}^{1/2}\|\patl_r^2\Del\psi(t,r,\cdot)\|_{L^2}^{1/2}\\
&\hspace{13mm}\|\patl_r\Del\patl_3\psi(t,r,\cdot)\|_{L^2}^{1/2}
\|\patl_r\Del\psi(t,r,\cdot)\|_{L^2}^{1/2}\big\}dr\\
\le&4\Big\{\int_{\R}\int^{R_0}_{r_0}(\patl_r^2\Del\patl_3\psi)^2rdrdx_3
\int_{\R}\int^{R_0}_{r_0}(\patl_r\Del\patl_3\psi)^2r^{-1}drdx_3\Big\}^{1/4}\\
&\hspace{5mm}\cdot\Big\{\int_{\R}\int^{R_0}_{r_0}(\patl_r^2\Del\psi)^2rdrdx_3
\int_{\R}\int^{R_0}_{r_0}(\patl_r\Del\psi)^2r^{-1}drdx_3\Big\}^{1/4}\\
\le&2\Big\{\int_{\R}\int^{R_0}_{r_0}(\Del_r\Del\patl_3\psi)^2rdrdx_3\Big\}^{1/2}
\Big\{\int_{\R}\int^{R_0}_{r_0}(\Del_r\Del\psi)^2rdrdx_3\Big\}^{1/2}.\\
\end{split}\eeq

There exists the following observation
\beq\label{Eqv-Obs-H5}\begin{split}
&\int_{\R}\int^{R_0}_{r_0}(\patl_r\Del^2\psi)^2rdrdx_3\\
=&\int_{\R}\int^{R_0}_{r_0}(\patl_r\Del_r\Del\psi)^2rdrdx_3
+\int_{\R}\int^{R_0}_{r_0}(\patl_r\patl_3^2\Del\psi)^2rdrdx_3\\
&+2\int_{\R}\int^{R_0}_{r_0}(\Del_r\patl_3\Del\psi)^2rdrdx_3.\\
\end{split}\eeq

Inserting \eqref{Est-LI-rz1} \eqref{Eqv-Obs-H3} \eqref{Eqv-Obs-H4}
\eqref{Eqv-Obs-H5} \eqref{Est-LI-r30} into \eqref{NT330}, we have
\beq\label{NT331}\begin{split}
&\Big|-\int_{\R}\int^{R_0}_{r_0}\{\Del_r\psi\patl_r\patl_3\Del\psi\}_t
\patl_r\Del\psi_trdrdx_3\\
&+\int_{\R}\int^{R_0}_{r_0}\{\patl_r(\frac1r\patl_r\Del\psi)
\patl_r\patl_3\psi\}_t\patl_r\Del\psi_tr^2drdx_3\Big|\\
\le&C\int_{\R}\int^{R_0}_{r_0}(\patl_r\Del\psi_t)^2rdrdx_3\\
&+C\Big\{\int_{\R}\int^{R_0}_{r_0}(\patl_r\Del^2\psi)^2rdrdx_3\Big\}^{1/4}
\Big\{\int_{\R}\int^{R_0}_{r_0}(\Del_r\psi_t)^2rdrdx_3\Big\}^{1/2}\\
&\hspace{5mm}\cdot\Big\{\int_{\R}\int^{R_0}_{r_0}(\patl_r\patl_3\Del
\psi_t)^2rdrdx_3\Big\}^{1/2}\\
&+C\Big\{\int_{\R}\int^{R_0}_{r_0}(\patl_r\Del^2\psi)^2rdrdx_3\Big\}^{1/4}
\Big\{\int_{\R}\int^{R_0}_{r_0}(\patl_r\patl_3\psi_t)^2rdrdx_3\Big\}^{1/2}\\
&\hspace{5mm}\cdot\Big\{\int_{\R}\int^{R_0}_{r_0}(\Del_r
\Del\psi_t)^2rdrdx_3\Big\}^{1/2}\\
\le&\frac{\nu}2\int_{\R}\int^{R_0}_{r_0}\{(\Del_r\Del\psi_t)^2
+(\patl_r\patl_3\Del\psi_t)^2\}rdrdx_3
+C\int_{\R}\int^{R_0}_{r_0}(\patl_r\Del\psi_t)^2rdrdx_3\\
&+C\int_{\R}\int^{R_0}_{r_0}\{(\Del_r\psi_t)^2+(\patl_r\patl_3\psi_t)^2\}rdrdx_3
\Big\{\int_{\R}\int^{R_0}_{r_0}(\patl_r\Del^2\psi)^2rdrdx_3\Big\}^{1/2},\\
\end{split}\eeq
where we have used the estimations \eqref{Est-H1-RCD00}
\eqref{Est-H2-RCD00}.

Inserting \eqref{Dt330} \eqref{Dr330} \eqref{NT331} into
\eqref{Est-H330}, we derive
\beq\label{Est-H31}\begin{split}
&\frac{d}{dt}\int_{\R}\int^{R_0}_{r_0}(\patl_r\Del\psi_t)^2rdrdx_3\\
&+\nu\int_{\R}\int^{R_0}_{r_0}\{(\Del_r\Del\psi_t)^2
+(\patl_r\patl_3\Del\psi_t)^2\}rdrdx_3\\
\le&C\int_{\R}\int^{R_0}_{r_0}(\patl_r\Del\psi_t)^2rdrdx_3\\
&+C\int_{\R}\int^{R_0}_{r_0}\{(\Del_r\psi_t)^2+(\patl_r\patl_3\psi_t)^2\}rdrdx_3
\Big\{\int_{\R}\int^{R_0}_{r_0}(\patl_r\Del^2\psi)^2rdrdx_3\Big\}^{1/2}.\\
\end{split}\eeq

By \eqref{CSR12-ps20-ph0} \eqref{Est-H20-t2r} \eqref{Est-H20-NL12}
\eqref{Est-H20-NL22} and \eqref{Est-H2-RCD00}, we obtain
\beq\label{Est-r5-r3t}\begin{split}
\nu^2\int_{\R}\int^{R_0}_{r_0}(\patl_r\Del^2\psi)^2rdrdx_3
\le3\int_{\R}\int^{R_0}_{r_0}(\patl_r\Del\psi_t)^2rdrdx_3+C.\\
\end{split}\eeq

Employing equation \eqref{CSR12-ps20-ph0} and using the same argument
as in the proof of \eqref{Est-H20-NL12} \eqref{Est-H20-NL22}, we get
\beq\label{Est-r5-r3t-I0}\begin{split}
\int_{\R}\int^{R_0}_{r_0}(\patl_r\Del\psi_t)^2(0,r,x_3)rdrdx_3
\le3\nu^2\int_{\R}\int^{R_0}_{r_0}(\patl_r\Del^2\psi_0)^2rdrdx_3+C.\\
\end{split}\eeq

Applying \eqref{Est-H31} \eqref{Est-r5-r3t} \eqref{Est-r5-r3t-I0}
\eqref{Est-r2t1-RCD00} and using Gronwall's inequality, we derive
\beq\label{Est-H32}\begin{split}
&\int_{\R}\int^{R_0}_{r_0}(\patl_r\Del\psi_t)^2rdrdx_3
+\int_{\R}\int^{R_0}_{r_0}(\patl_r\Del^2\psi)^2rdrdx_3\\
&+\int^t_0\int_{\R}\int^{R_0}_{r_0}\{(\Del_r\Del\psi_s)^2
+(\patl_r\patl_3\Del\psi_s)^2\}(s,r,x_3)rdrdx_3\\
\le&C,\;\;\;\;\forall t\ge0.\\
\end{split}\eeq

By \eqref{CSR12-ps20-ph0}, we have
\beq\label{Eq1-r4t-r6}\begin{split}
\nu\Del_r\Del^2\psi
=&\Del_r\Del\psi_t-\patl_r\{\Del_r\psi\patl_r\patl_3\Del
\psi\}-r^{-1}\Del_r\psi\patl_r\patl_3\Del\psi\\
&+\patl_r\{\patl_r(\frac1r\patl_r\Del\psi)\patl_r\patl_3\psi\}r
+2\patl_r(\frac1r\patl_r\Del\psi)\patl_r\patl_3\psi,
\end{split}\eeq

\beq\label{Eq2-r4t-r5z}\begin{split}
\nu\patl_r\patl_3\Del^2\psi
=\patl_r\patl_3\Del\psi_t-\patl_3\{\Del_r\psi\patl_r\patl_3\Del
\psi\}+\patl_3\{\patl_r(\frac1r\patl_r\Del\psi)\patl_r\patl_3\psi\}r.
\end{split}\eeq
Similar the proof of \eqref{Est-r5-r3t}, by equations
\eqref{Eq1-r4t-r6} \eqref{Eq2-r4t-r5z} and H$\ddot{o}$lder's
inequality, it is easy to prove
\beq\label{Est-r4t-r6}\begin{split}
&\nu^2\int_{\R}\int^{R_0}_{r_0}\{(\Del_r\Del^2\psi)^2
+(\patl_r\patl_3\Del^2\psi)^2\}rdrdx_3\\
\le&6\int_{\R}\int^{R_0}_{r_0}\{(\Del_r\Del\psi_t)^2
+(\patl_r\patl_3\Del\psi_t)^2\}rdrdx_3\\
&+6\int_{\R}\int^{R_0}_{r_0}\Big\{(\patl_r\{\Del_r\psi\patl_r\patl_3
\Del\psi\})^2+(r^{-1}\Del_r\psi\patl_r\patl_3\Del\psi)^2\\
&+\big(\patl_r\{\patl_r(\frac1r\patl_r\Del\psi)\patl_r\patl_3\psi\}r\big)^2
+\{\patl_r(\frac1r\patl_r\Del\psi)\patl_r\patl_3\psi\}^2\\
&+(\patl_3\{\Del_r\psi\patl_r\patl_3\Del\psi\})^2
+\big(\patl_3\{\patl_r(\frac1r\patl_r\Del\psi)\patl_r\patl_3\psi\}r\big)^2\Big\}rdrdx_3.\\
\end{split}\eeq
By using \eqref{Est-H2-RCD00} \eqref{Est-H32} \eqref{Eqv-Obs-H4}
\eqref{Eqv-Obs-H5} and Gagliardo-Nirenberg's inequality, we get
\beq\label{Est-r4t-NT1}\begin{split}
&\int_{\R}\int^{R_0}_{r_0}(\patl_r\{\Del_r\psi\patl_r\patl_3
\Del\psi\})^2 rdrdx_3\\
\le&\|\Del_r\psi(t,\cdot)\|_{L^{\infty}}\int_{\R}\int^{R_0}_{r_0}
(\patl_r^2\patl_3\Del\psi)^2 rdrdx_3\\
&+\|\patl_r\Del_r\psi(t,\cdot)\|_{L^{\infty}}\int_{\R}\int^{R_0}_{r_0}
(\patl_r\patl_3\Del\psi)^2 rdrdx_3\\
\le&C,
\end{split}\eeq

\beq\label{Est-r4t-NT2}\begin{split}
&\int_{\R}\int^{R_0}_{r_0}(r^{-1}\Del_r\psi\patl_r\patl_3\Del\psi)^2
rdrdx_3\\
\le&C\|\Del_r\psi(t,\cdot)\|_{L^{\infty}}\int_{\R}\int^{R_0}_{r_0}
(\patl_r\patl_3\Del\psi)^2rdrdx_3\le C,\\
\end{split}\eeq

\beq\label{Est-r4t-NT3}\begin{split}
&\int_{\R}\int^{R_0}_{r_0}\big(\patl_r\{\patl_r(\frac1r\patl_r\Del\psi)
\patl_r\patl_3\psi\}r\big)^2 rdrdx_3\\
\le&C\|\patl_r\patl_3\psi(t,\cdot)\|_{L^{\infty}}\int_{\R}
\int^{R_0}_{r_0}\{(\patl_r^3\Del\psi)^2+(\patl_r^2\Del\psi)^2
+(\patl_r\Del\psi)^2\}rdrdx_3\\
&+C\|\patl_r^2\patl_3\psi(t,\cdot)\|_{L^{\infty}}\int_{\R}
\int^{R_0}_{r_0}\{(\patl_r^2\Del\psi)^2+(\patl_r\Del\psi)^2\}rdrdx_3\\
\le&C,
\end{split}\eeq

\beq\label{Est-r4t-NT4}\begin{split}
&\int_{\R}\int^{R_0}_{r_0}\{\patl_r(\frac1r\patl_r\Del\psi)
\patl_r\patl_3\psi\}^2 rdrdx_3\\
\le&C\|\patl_r\patl_3\psi(t,\cdot)\|_{L^{\infty}}\int_{\R}
\int^{R_0}_{r_0}\{(\patl_r^2\Del\psi)^2+(\patl_r\Del\psi)^2\}rdrdx_3\\
\le&C,
\end{split}\eeq

\beq\label{Est-r4t-NT5}\begin{split}
&\int_{\R}\int^{R_0}_{r_0}(\patl_3\{\Del_r\psi\patl_r\patl_3
\Del\psi\})^2 rdrdx_3\\
\le&C\|\Del_r\psi(t,\cdot)\|_{L^{\infty}}\int_{\R}\int^{R_0}_{r_0}
(\patl_r\patl_3^2\Del\psi)^2rdrdx_3\\
&+C\|\patl_3\Del_r\psi(t,\cdot)\|_{L^{\infty}}\int_{\R}\int^{R_0}_{r_0}
(\patl_r\patl_3\Del\psi)^2rdrdx_3\\
\le&C,\\
\end{split}\eeq

\beq\label{Est-r4t-NT6}\begin{split}
&\int_{\R}\int^{R_0}_{r_0}\big(\patl_3\{\patl_r(\frac1r\patl_r\Del\psi)
\patl_r\patl_3\psi\}r\big)^2 rdrdx_3\\
\le&C\|\patl_r\patl_3\psi(t,\cdot)\|_{L^{\infty}}\int_{\R}\int^{R_0}_{
r_0}\{(\patl_3\patl_r^2\Del\psi)^2+(\patl_r\patl_3\Del\psi)^2\}rdrdx_3\\
&+C\|\patl_r\patl_3^2\psi(t,\cdot)\|_{L^{\infty}}\int_{\R}
\int^{R_0}_{r_0}\{(\patl_r^2\Del\psi)^2+(\patl_r\Del\psi)^2\}rdrdx_3\\
\le&C.
\end{split}\eeq
Inserting \eqref{Est-r4t-NT1}--\eqref{Est-r4t-NT6} into
\eqref{Est-r4t-r6}, we obtain
\beq\label{Est-r4t-r61}\begin{split}
&\nu^2\int_{\R}\int^{R_0}_{r_0}\{(\Del_r\Del^2\psi)^2
+(\patl_r\patl_3\Del^2\psi)^2\}rdrdx_3\\
\le&6\int_{\R}\int^{R_0}_{r_0}\{(\Del_r\Del\psi_t)^2
+(\patl_r\patl_3\Del\psi_t)^2\}rdrdx_3+C.\\
\end{split}\eeq

Putting \eqref{Est-r5-r3t} \eqref{Est-r4t-r61} together,
we derive \eqref{Est-H3-RCD00}.
\end{proof}

\begin{lemma}\label{Est-H4-RCD} For the solutions of problem
\eqref{CSR12-ps20-ph0} \eqref{IDa-ps} \eqref{BC-RCD00}, we have
\beq\label{Est-H4-RCD00}\begin{split}
&\int_{\R}\int^{R_0}_{r_0}\{(\Del_r\Del\psi_t)^2
+(\patl_r\patl_3\Del\psi_t)^2\}rdrdx_3\\
&\int_{\R}\int^{R_0}_{r_0}\{(\Del_r\Del^2\psi)^2
+(\patl_r\patl_3\Del^2\psi)^2\}rdrdx_3\\
&+\int^t_0\int_{\R}\int^{R_0}_{r_0}\{\patl_r\Del\psi_{tt}\}^2
(s,r,x_3)rdrdx_3ds\\
&+\int^t_0\int_{\R}\int^{R_0}_{r_0}\{\patl_r\Del^2\psi_t\}^2
(s,r,x_3)rdrdx_3ds\\
\le&C_4,\;\;\;\;\;\;\forall t\ge0,\\
\end{split}\eeq
where positive constant $C_4$ depends on $t$, $\nu$, $r_0$, $R_0$ and
\[\begin{split}
&\int_{\R}\int^{R_0}_{r_0}(\patl_r\Del^m\psi_0)^2rdrdx_3\;\;(m=1,2),\;\;
\int_{\R}\int^{R_0}_{r_0}(\patl_r\Del\psi_0)^2r^{-1}drdx_3,\\
&\int_{\R}\int^{R_0}_{r_0}(\Del_r\Del^k\psi_0)^2rdrdx_3
+\int_{\R}\int^{R_0}_{r_0}(\patl_r\patl_3\Del^k\psi_0)^2rdrdx_3\;\;(k=0,1,2).\\
\end{split}\]
\end{lemma}

\begin{proof} Taking inner product of the equation
\eqref{CSR12-ps20-t2} with $r\patl_r\Del\psi_{tt}$, we have
\beq\label{Est-H40}\begin{split}
&\int_{\R}\int^{R_0}_{r_0}\patl_r\Del(\psi_{tt}-\nu\Del\psi_t)
\patl_r\Del\psi_{tt}rdrdx_3\\
&-\int_{\R}\int^{R_0}_{r_0}\{\Del_r\psi\patl_r\patl_3\Del\psi\}_t
\patl_r\Del\psi_{tt}rdrdx_3\\
&+\int_{\R}\int^{R_0}_{r_0}\{\patl_r(\frac1r\patl_r\Del\psi)
\patl_r\patl_3\psi\}_t\patl_r\Del\psi_{tt}r^2drdx_3=0.
\end{split}\eeq

By part integration, we derive
\beq\label{H4-r4t1-dt}\begin{split}
&-\nu\int_{\R}\int^{R_0}_{r_0}\patl_r\Del^2\psi_t\patl_r\Del\psi_{tt}rdrdx_3\\
=&\frac{\nu}2\frac{d}{dt}\int_{\R}\int^{R_0}_{r_0}
(\Del_r\Del\psi_t)^2rdrdx_3\\
&+\frac{\nu}2\frac{d}{dt}\int_{\R}\int^{R_0}_{r_0}
(\patl_r\patl_3\Del\psi_t)^2rdrdx_3.\\
\end{split}\eeq

By H$\ddot{o}$lder's inequality, we obtain
\beq\label{H4-NT10}\begin{split}
&\Big|\int_{\R}\int^{R_0}_{r_0}\{\Del_r\psi\patl_r\patl_3\Del\psi\}_t
\patl_r\Del\psi_{tt}rdrdx_3\Big|\\
\le&\Big\{\|\Del_r\psi(t,\cdot)\|_{L^{\infty}}\Big(\int_{\R}
\int^{R_0}_{r_0}(\patl_r\patl_3\Del\psi_t)^2rdrdx_3\Big)^{1/2}\\
&+\|\patl_r\patl_3\Del\psi(t,\cdot)\|_{L^{\infty}}\Big(\int_{\R}
\int^{R_0}_{r_0}(\Del_r\psi_t)^2rdrdx_3\Big)^{1/2}\Big\}\\
&\hspace{35mm}\cdot\Big(\int_{\R}\int^{R_0}_{r_0}
(\patl_r\Del\psi_{tt})^2rdrdx_3\Big)^{1/2},\\
\end{split}\eeq

\beq\label{H4-NT20}\begin{split}
&\Big|\int_{\R}\int^{R_0}_{r_0}\{\patl_r(\frac1r\patl_r\Del\psi)
\patl_r\patl_3\psi\}_t\patl_r\Del\psi_{tt}r^2drdx_3\Big|\\
\le&\Big\{\|\patl_r\patl_3\psi(t,\cdot)\|_{L^{\infty}}\Big(\int_{\R}
\int^{R_0}_{r_0}(\Del_r\Del\psi_t)^2rdrdx_3\Big)^{1/2}\\
&+\|\patl_r\patl_3\psi_t(t,\cdot)\|_{L^{\infty}}\Big(\int_{\R}
\int^{R_0}_{r_0}(\Del_r\Del\psi)^2rdrdx_3\Big)^{1/2}\Big\}\\
&\hspace{35mm}\cdot\Big(\int_{\R}\int^{R_0}_{r_0}
(\patl_r\Del\psi_{tt})^2rdrdx_3\Big)^{1/2}.\\
\end{split}\eeq
In \eqref{H4-NT20}, the following observation \eqref{Eqv-Obs-r4t1} has been used
\beq\label{Eqv-Obs-r4t1}\begin{split}
&\int_{\R}\int^{R_0}_{r_0}(\Del_r\Del\psi_t)^2rdrdx_3\\
=&\int_{\R}\int^{R_0}_{r_0}\{(\patl_r^2\Del\psi_t)^2r
+(\patl_r\Del\psi_t)^2r^{-1}\}drdx_3.\\
\end{split}\eeq.

By the same argument as in the proof of \eqref{Est-LI-rz1}, we get
\beq\label{Est-LI-rz-t1}\begin{split}
&\|\patl_r\patl_3\psi_t(t,\cdot)\|_{L^{\infty}}^2
=\|(\patl_r\patl_3\psi_t)^2(t,\cdot)\|_{L^{\infty}}\\
\le&4\int^{R_0}_{r_0}\big\{\|\patl_r^2\patl_3^2\psi_t(t,r,\cdot)\|_{L^2}^{1/2}
\|\patl_r^2\patl_3\psi_t(t,r,\cdot)\|_{L^2}^{1/2}\\
&\hspace{13mm}\|\patl_r\patl_3^2\psi_t(t,r,\cdot)\|_{L^2}^{1/2}
\|\patl_r\patl_3\psi_t(t,r,\cdot)\|_{L^2}^{1/2}\big\}dr\\
\le&2\Big\{\int_{\R}\int^{R_0}_{r_0}(\Del_r\patl_3^2\psi_t)^2rdrdx_3\Big\}^{1/2}
\Big\{\int_{\R}\int^{R_0}_{r_0}(\Del_r\patl_3\psi_t)^2rdrdx_3\Big\}^{1/2}\\
\le&2\Big\{\int_{\R}\int^{R_0}_{r_0}(\Del_r\Del\psi_t)^2rdrdx_3\Big\}^{1/2}
\Big\{\int_{\R}\int^{R_0}_{r_0}(\patl_r\Del\psi_t)^2rdrdx_3\Big\}^{1/2},\\
\end{split}\eeq

\beq\label{Est-LI-rz-d1}\begin{split}
&\|\patl_r\patl_3\Del\psi(t,\cdot)\|_{L^{\infty}}^2
=\|(\patl_r\patl_3\Del\psi)^2(t,\cdot)\|_{L^{\infty}}\\
\le&4\int^{R_0}_{r_0}\big\{\|\patl_r^2\patl_3^2\Del\psi(t,r,\cdot)\|_{L^2}^{1/2}
\|\patl_r^2\patl_3\Del\psi(t,r,\cdot)\|_{L^2}^{1/2}\\
&\hspace{13mm}\|\patl_r\patl_3^2\Del\psi(t,r,\cdot)\|_{L^2}^{1/2}
\|\patl_r\patl_3\Del\psi(t,r,\cdot)\|_{L^2}^{1/2}\big\}dr\\
\le&2\Big\{\int_{\R}\int^{R_0}_{r_0}(\Del_r\patl_3^2\Del\psi)^2rdrdx_3\Big\}^{1/2}
\Big\{\int_{\R}\int^{R_0}_{r_0}(\Del_r\patl_3\Del\psi)^2rdrdx_3\Big\}^{1/2}.\\
\le&2\Big\{\int_{\R}\int^{R_0}_{r_0}(\Del_r\Del^2\psi)^2rdrdx_3\Big\}^{1/2}
\Big\{\int_{\R}\int^{R_0}_{r_0}(\patl_r\Del^2\psi)^2rdrdx_3\Big\}^{1/2}.\\
\end{split}\eeq
Here \eqref{Eqv-Obs-H3} \eqref{Eqv-Obs-H4} are used.

Inserting \eqref{Est-LI-rz1} \eqref{Est-LI-rr0} \eqref{Est-LI-rz-t1}
\eqref{Est-LI-rz-d1} into \eqref{H4-NT10} \eqref{H4-NT20}, we have
\beq\label{H4-NT120}\begin{split}
&\Big|-\int_{\R}\int^{R_0}_{r_0}\{\Del_r\psi\patl_r\patl_3\Del\psi\}_t
\patl_r\Del\psi_{tt}rdrdx_3\\
&+\int_{\R}\int^{R_0}_{r_0}\{\patl_r(\frac1r\patl_r\Del\psi)
\patl_r\patl_3\psi\}_t\patl_r\Del\psi_{tt}r^2drdx_3\Big|\\
\le&\frac12\int_{\R}\int^{R_0}_{r_0}(\patl_r\Del\psi_{tt})^2rdrdx_3
+C\int_{\R}\int^{R_0}_{r_0}(\patl_r\patl_3\Del\psi_t)^2rdrdx_3\\
&+C\int_{\R}\int^{R_0}_{r_0}(\Del_r\psi_t)^2rdrdx_3\Big\{\int_{\R}
\int^{R_0}_{r_0}(\Del_r\Del^2\psi)^2rdrdx_3\Big\}^{1/2}\\
&+C\int_{\R}\int^{R_0}_{r_0}(\Del_r\Del\psi_t)^2rdrdx_3+C,\\
\end{split}\eeq
where we have used the estimates \eqref{Est-H1-RCD00} \eqref{Est-H2-RCD00}
\eqref{Est-H3-RCD00}.

Inserting \eqref{H4-r4t1-dt} \eqref{H4-NT120} into \eqref{Est-H40},
we obtain
\beq\label{Est-H41}\begin{split}
&\nu\frac{d}{dt}\int_{\R}\int^{R_0}_{r_0}\{(\Del_r\Del\psi_t)^2
+(\patl_r\patl_3\Del\psi_t)^2\}rdrdx_3\\
&+\int_{\R}\int^{R_0}_{r_0}(\patl_r\Del\psi_{tt})^2rdrdx_3\\
\le&C\int_{\R}\int^{R_0}_{r_0}\{(\Del_r\Del\psi_t)^2+
(\patl_r\patl_3\Del\psi_t)^2\}rdrdx_3+C\\
&+C\int_{\R}\int^{R_0}_{r_0}(\Del_r\psi_t)^2rdrdx_3\Big\{\int_{\R}
\int^{R_0}_{r_0}(\Del_r\Del^2\psi)^2rdrdx_3\Big\}^{1/2}.\\
\end{split}\eeq
By the same argument as in the proof of \eqref{Est-r4t-r61}, we get
\beq\label{Est-r4t-r6t0}\begin{split}
&\int_{\R}\int^{R_0}_{r_0}\{(\Del_r\Del\psi_t)^2
+(\patl_r\patl_3\Del\psi_t)^2\}(0,r,x_3)rdrdx_3\\
\le&3\nu^2\int_{\R}\int^{R_0}_{r_0}(\Del_r\Del^2\psi_0)^2rdrdx_3+C.\\
\end{split}\eeq
Employing \eqref{Est-r2t1-RCD00} \eqref{Est-r4t-r61} \eqref{Est-H41}
\eqref{Est-r4t-r6t0} and Gronwall's inequality, we drive
\beq\label{Est-H42}\begin{split}
&\int_{\R}\int^{R_0}_{r_0}\{(\Del_r\Del\psi_t)^2
+(\patl_r\patl_3\Del\psi_t)^2\}rdrdx_3\\
&+\int_{\R}\int^{R_0}_{r_0}\{(\Del_r\Del^2\psi)^2
+(\patl_r\patl_3\Del^2\psi)^2\}rdrdx_3\\
&+\int^t_0\int_{\R}\int^{R_0}_{r_0}\{\patl_r\Del\psi_{ss}\}^2
(s,r,x_3)rdrdx_3\\
\le&C,\;\;\;\;\forall t\ge0.\\
\end{split}\eeq

Applying equation \eqref{CSR12-ps20-t2} and estimates \eqref{Est-H1-RCD00}
\eqref{Est-H2-RCD00} \eqref{Est-H3-RCD00}\eqref{Est-H42}, it is
easy to prove
\beq\label{Est-r3t2-r5t1}\begin{split}
\nu^2\int_{\R}\int^{R_0}_{r_0}(\patl_r\Del^2\psi_t)^2rdrdx_3
\le3\int_{\R}\int^{R_0}_{r_0}(\patl_r\Del\psi_{tt})^2rdrdx_3+C.\\
\end{split}\eeq

By \eqref{Est-H42} and \eqref{Est-r3t2-r5t1}, \eqref{Est-H4-RCD00}
is proved.
\end{proof}

Differentiating equation \eqref{CSR12-ps20-ph0} with $\patl_t^m$,
we obtian
\beq\label{CSR12-ps20-tm}\begin{split}
\patl_t^m\patl_r\Del(\psi_t-\nu\Del\psi)
-\patl_t^m\{\Del_r\psi\patl_r\patl_3\Del\psi\}
+\patl_t^m\{\patl_r(\frac1r\patl_r\Del\psi)\patl_r\patl_3\psi\}r=0.
\end{split}\eeq
Then taking inner product of the equation \eqref{CSR12-ps20-tm} with
$r\patl_r\Del\patl_t^m\psi$, $r\patl_r\Del\patl_t^{m+1}\psi$ and
$r\patl_r\patl_t^{m+1}\psi$ respectively, by induction with respect
to $m$, we can derive more regular estimations.

By Lemma \ref{Est-L2-RCD}--\ref{Est-H4-RCD}, the following result
is proved.

\begin{theorem}[Global Solution in Ring Cylinder Domain]
\label{CSR12-RCD-Thm-ps} Provided that $m\ge4$,
\[
u_0=\big(\patl_1
\patl_3\psi_0,\patl_2\patl_3\psi_0,-(\patl_1^2+\patl_2^2)\psi_0\big),
\]
$\psi_0=\psi_0(r,x_3)$ satisfies $\patl_r\psi_0|_{\patl{\mathbb D}}=
\patl_r\Del\psi_0|_{\patl{\mathbb D}}=0$ and
\beq\label{IDa0-RCD-psHm}\begin{split}
&\sum_{j\ge0,\;2j\le m}\int_{\R}\int^{R_0}_{r_0}\{(\Del_r\Del^j\psi_0)^2
+(\patl_r\patl_3\Del^j\psi_0)^2\}rdrdx_3\\
&+\sum_{k\ge0,\;2k+1\le m}\int_{\R}\int^{R_0}_{r_0}(\patl_r\Del\Del^k
\psi_0)^2rdrdx_3\\
&+\int_{\R}\int^{R_0}_{r_0}(\patl_r\Del\psi_0)^2r^{-1}drdx_3<\infty.\\
\end{split}\eeq
Then there exists a unique global solution $\psi$ of the problem
\eqref{CSR12-ps20-ph0} \eqref{IDa-ps} \eqref{BC-RCD00} such that
\beq\label{Est-RCD-psH5}\begin{split}
&\sum^2_{k=0}\int_{\R}\int^{R_0}_{r_0}\{(\Del_r\Del^k\psi)^2
+(\patl_r\patl_3\Del^k\psi)^2\}rdrdx_3\\
&+\sum^1_{k=0}\int_{\R}\int^{R_0}_{r_0}
(\patl_r\Del\Del^k\psi)^2rdrdx_3\\
&+\int_{\R}\int^{R_0}_{r_0}(\patl_r\Del\psi)^2r^{-1}drdx_3
\le C(\psi_0,\nu,t)<\infty,\;\;\;\;\;\;\forall t\ge0.\\
\end{split}\eeq
\end{theorem}

\section{Fluid in ${\mathbb{R}}^3$}
\setcounter{equation}{0}

This section is devoted to consider the problem \eqref{NS1}--
\eqref{NSi} and prove Theorem \ref{CSR12-Thm-u}. Provided $\psi_0$
regular enough. If necessary, let $\psi_0$ be modified.

Take functions $\chi^n(r)\in C^{\infty}_0([0,\infty))$ $(n=1,2,\cdots)$
such that $0\le\chi^n(r)\le1$ and
\beq\label{Def-chi}
\chi^n(r)=\left\{\begin{array}{ll}1,&\frac1n\le r\le n\\
0,&0\le r\le\frac1{2n}\;\mbox{ or }\;2n\le r<\infty\end{array}\right..
\eeq
Moreover there exists positive constant $C$, which is independent of
$n$, such that
\beq\label{Est-chi}\begin{split}
&|\patl_r\chi^n(r)|\le\left\{\begin{array}{ll}Cn,&\frac1{2n}<r<\frac1n\\
\frac{C}{2n},&n<r<2n\end{array}\right.,\\
&r|\patl_r\chi^n(r)|\le C,\;\;\;\;\forall r\ge0,\;\;\forall n\ge1.\\
\end{split}\eeq

Let ${\mathbb D}^n={\mathbb D}\big|_{r_0=\frac1{2n}, R_0=2n}$ and
$\psi_0^n=\psi_0^n(r,x_3)$ be solution of the following problem
\beq\label{Eq-psi-n0}
\Del\psi_0^n=\int^r_{\frac1{2n}}\chi^n(s)\patl_s\Del\psi_0(s,x_3)ds,
\;\;\;\;\;\;(r,x_3)\in {\mathbb D}^n,
\eeq
\beq\label{BC-psi-n0}
\patl_r\psi_0^n\big|_{\patl{\mathbb D}^n}=0.
\eeq
Equation \eqref{Eq-psi-n0} is equivalent to the following equation
\beq\label{Eq-psi-n01}
\patl_r\Del\psi_0^n=\chi^n(r)\patl_r\Del\psi_0(r,x_3),
\;\;\;\;\;\;(r,x_3)\in {\mathbb D}^n.
\eeq

Taking inner product of equation \eqref{Eq-psi-n01} with $-r\patl_r
\psi_0^n$, we have
\beq\label{Est-L2-psi-n0}\begin{split}
&\int_{{\mathbb D}^n}\{(\Del_r\psi_0^n)^2+(\patl_r\patl_3\psi_0^n)^2\}
rdrdx_3\\
=&-\int_{{\mathbb D}^n}\chi^n\patl_r\Del\psi_0\patl_r\psi_0^nrdrdx_3\\
=&\int_{{\mathbb D}^n}\chi^n\{\Del_r\psi_0\Del_r\psi_0^n+
\patl_r\patl_3\psi_0\patl_r\patl_3\psi_0^n\}rdrdx_3\\
&+\int_{{\mathbb D}^n}\Del_r\psi_0\patl_r\psi_0^nr\patl_r\chi^ndrdx_3.\\
\end{split}\eeq
By H$\ddot{o}$lder's inequality, we obtain
\beq\label{Est-psn0-ps0-1}\begin{split}
&\Big|\int_{{\mathbb D}^n}\chi^n\{\Del_r\psi_0\Del_r\psi_0^n+
\patl_r\patl_3\psi_0\patl_r\patl_3\psi_0^n\}rdrdx_3\Big|\\
\le&\frac14\int_{{\mathbb D}^n}\{(\Del_r\psi_0^n)^2
+(\patl_r\patl_3\psi_0^n)^2\}rdrdx_3\\
&+\int_{{\mathbb D}^n}(\chi^n)^2\{(\Del_r\psi_0)^2
+(\patl_r\patl_3\psi_0)^2\}rdrdx_3\\
\le&\frac14\int_{{\mathbb D}^n}\{(\Del_r\psi_0^n)^2
+(\patl_r\patl_3\psi_0^n)^2\}rdrdx_3\\
&+\int_{\R}\int^{\infty}_0\{(\Del_r\psi_0)^2
+(\patl_r\patl_3\psi_0)^2\}rdrdx_3,\\
\end{split}\eeq
\beq\label{Est-psn0-ps0-2}\begin{split}
&\Big|\int_{{\mathbb D}^n}\Del_r\psi_0\patl_r\psi_0^nr
\patl_r\chi^ndrdx_3\Big|\\
\le&\frac14\int_{{\mathbb D}^n}(\patl_r\psi_0^n)^2r^{-1}drdx_3
+\int_{{\mathbb D}^n}(\Del_r\psi_0)^2(r\patl_r\chi^n)^2rdrdx_3\\
\le&\frac14\int_{{\mathbb D}^n}(\Del_r\psi_0^n)^2rdrdx_3
+C\int_{\R}\int^{\infty}_0(\Del_r\psi_0)^2rdrdx_3,\\
\end{split}\eeq
where the estimate \eqref{Est-chi} has been used. Inserting estimates
\eqref{Est-psn0-ps0-1} \eqref{Est-psn0-ps0-2} into
\eqref{Est-L2-psi-n0}, we get
\beq\label{Est-IDa0n-L2}\begin{split}
&\int_{{\mathbb D}^n}\{(\Del_r\psi_0^n)^2+(\patl_r\patl_3\psi_0^n)^2\}
rdrdx_3\\
\le&C\int_{\R}\int^{\infty}_0\{(\Del_r\psi_0)^2
+(\patl_r\patl_3\psi_0)^2\}rdrdx_3,\\
\end{split}\eeq
where constant $C$ is independent of $n$.

By equation \eqref{Eq-psi-n01}, we derive
\beq\label{Est-IDa0n-H1}\begin{split}
&\int_{{\mathbb D}^n}(\patl_r\Del\psi_0^n)^2rdrdx_3
\le\int_{\R}\int^{\infty}_0(\patl_r\Del\psi_0)^2rdrdx_3,\\
\end{split}\eeq
\beq\label{Est-IDa0n-H1w}\begin{split}
&\int_{{\mathbb D}^n}(\patl_r\Del\psi_0^n)^2r^{-1}drdx_3
\le\int_{\R}\int^{\infty}_0(\patl_r\Del\psi_0)^2r^{-1}drdx_3.\\
\end{split}\eeq

We consider the following problem
\beq\label{NSEq-psin}\begin{split}
\patl_r\Del(\psi_t^n-\nu\Del\psi^n)
-\Del_r\psi^n\cdot\patl_r\patl_3\Del\psi^n
+r\patl_r(\frac1r\patl_r\Del\psi^n)\cdot\patl_r\patl_3\psi^n=0,
\end{split}\eeq
\beq\label{NSIDa-psin0}
\psi^n(t,r,x_3)|_{t=0}=\psi_0^n(r,x_3),
\eeq
\beq\label{NSBC-psi-n0}
\patl_r\psi^n\big|_{\patl{\mathbb D}^n}=0,\;\;\;\;\;\;
\patl_r\Del\psi^n\big|_{\patl{\mathbb D}^n}=0.
\eeq
By Theorem \ref{CSR12-RCD-Thm-ps}, there exists a unique global regular
solution $\psi^n$ of the problem \eqref{NSEq-psin}--\eqref{NSBC-psi-n0}.
Moreover by Lemma \ref{Est-L2-RCD}--\ref{Est-H1-RCD}, we have
\beq\label{Est-psin-L2H1}\begin{split}
&\int_{{\mathbb D}^n}\{(\Del_r\psi^n)^2+(\patl_r\patl_3\psi^n)^2
+(\patl_r\Del\psi^n)^2+(\patl_r\Del\psi^n)^2r^{-2}\}rdrdx_3\\
&+\int^t_0\int_{{\mathbb D}^n}\{(\patl_r\Del\psi^n)^2
+(\Del_r\Del\psi^n)^2+(\patl_r\patl_3\Del\psi^n)^2\}rdrdx_3ds\\
\le&C(T,\psi_0,\nu),\;\;\;\;\;\;\forall T\ge0,\;\;0\le t\le T,\\
\end{split}\eeq
where constant $C(T,\psi_0,\nu)$ is independent of $n$.

Let
\beq\label{Def-un}\begin{split}
u^n(t,x)=\big(\patl_1\patl_3\psi^n,\patl_2\patl_3\psi^n,
-(\patl_1^2+\patl_2^2)\psi^n\big).\\
\end{split}\eeq
Then $u^n(t,x)$ is solution of the following problem
\beq\label{NS1-un}
u_t^n-\nu\Del u^n+(u^n\cdot\nab)u^n+\nab P=0,
\;\;\;\;\;\;x\in{\mathbb D}^n,
\eeq
\beq\label{NS2-un}
\nab\cdot u^n=0,
\eeq
\beq\label{NSi-un}
u^n|_{t=0}=u_0^n=\big(\patl_1\patl_3\psi_0^n,\patl_2\patl_3\psi_0^n,
-(\patl_1^2+\patl_2^2)\psi_0^n\big).
\eeq

By straightforward calculation, we have
\beq\begin{split}\label{EqvN-ups}
&\|u^n(t,\cdot)\|_{L^2({\mathbb D}^n)}^2=2\pi\int_{{\mathbb D}^n}
\{(\Del_r\psi^n)^2+(\patl_r\patl_3\psi^n)^2\}rdrdx_3,\\
&\|\nab u^n(t,\cdot)\|_{L^2({\mathbb D}^n)}^2=2\pi\int_{{\mathbb D}^n}
(\patl_r\Del\psi^n)^2rdrdx_3,\\
&\|\Del u^n(t,\cdot)\|_{L^2({\mathbb D}^n)}^2=2\pi\int_{{\mathbb D}^n}
\{(\Del_r\Del\psi^n)^2+(\patl_r\patl_3\Del\psi^n)^2\}rdrdx_3.\\
\end{split}\eeq
\eqref{Est-psin-L2H1} and \eqref{EqvN-ups} mean that
\beq\label{Est-un-L2H1}\begin{split}
&\|u^n(t,\cdot)\|_{L^2({\mathbb D}^n)}^2
+\|\nab u^n(t,\cdot)\|_{L^2({\mathbb D}^n)}^2\\
&+\int^t_0\{\|\nab u^n(s,\cdot)\|_{L^2({\mathbb D}^n)}^2
+\|\Del u^n(s,\cdot)\|_{L^2({\mathbb D}^n)}^2\}ds\\
\le&C(T,\psi_0,\nu),\;\;\;\;\;\;\forall T\ge0,\;\;0\le t\le T,\\
\end{split}\eeq
where constant $C(T,\psi_0,\nu)$ is independent of $n$.

Applying the equation \eqref{NS1-un} and estimate \eqref{Est-un-L2H1},
we derive
\beq\label{Est-unt-L2}\begin{split}
\int^t_0\|u^n_s(s,\cdot)\|_{L^2({\mathbb D}^n)}^2ds
\le C(T,\psi_0,\nu),\;\;\;\;\;\;\forall T\ge0,\;\;0\le t\le T,\\
\end{split}\eeq
where constant $C(T,\psi_0,\nu)$ is independent of $n$.

Let $n\rightarrow\infty$. By using estimates \eqref{Est-un-L2H1}
\eqref{Est-unt-L2} and Aubin-Lions's Lemma, there exist $u\in
L^{\infty}([0,T];H^1(\R^3)\cap H(\R^3))\cap L^2([0,T];H^2(\R^3))$ and
sub-sequence $\{u^{n^{\prime}}\}$ of sequence $\{u^n\}$ such that
\beq\label{Lim-un-L2}\begin{split}
u^{n^{\prime}}\rightarrow u\;\;\mbox{ as }\;\;n^{\prime}\rightarrow\infty,
\end{split}\eeq
weakly-$\ast$ in $L^{\infty}([0,T];H^1({\mathbb D}))$, weakly in $L^2(
[0,T];H^2({\mathbb D}))$ and strongly in $L^2([0,T];H^1({\mathbb D}))$
for any $T\ge0$ and any ${\mathbb D}$. Moreover
\beq\label{Lim-unt-L2}\begin{split}
u^{n^{\prime}}_t\rightarrow u_t\;\;\mbox{ as }\;\;n^{\prime}\rightarrow\infty,
\end{split}\eeq
weakly in $L^2([0,T];L^2({\mathbb D}))$ for any $T\ge0$ and any
${\mathbb D}$. By Proposition \ref{NS-LHmS},
this $u$ is the unique global solution of problem \eqref{NS1}--
\eqref{NSi} and $u\in C([0,T];H^1(\R^3)\cap H(\R^3))$ for any $T\ge0$.
Further more the following estimate holds
\beq\label{Est-u-H1}\begin{split}
&\|u(t,\cdot)\|_{L^2(\R^3)}^2
+\|\nab u(t,\cdot)\|_{L^2(\R^3)}^2\\
&+\int^t_0\{\|\nab u(s,\cdot)\|_{L^2(\R^3)}^2
+\|\Del u(s,\cdot)\|_{L^2(\R^3)}^2+\|u_s(s,\cdot)\|_{L^2(\R^3)}^2\}ds\\
\le&C(T,u_0,\nu),\;\;\;\;\;\;\forall T\ge0,\;\;0\le t\le T.\\
\end{split}\eeq

\eqref{Est-psin-L2H1} and \eqref{Def-un} imply that there exists
$\psi(r,x_3)$ such that
\beq\label{Def-u-psi}\begin{split}
u(t,x)=\big(\patl_1\patl_3\psi,\patl_2\patl_3\psi,
-(\patl_1^2+\patl_2^2)\psi\big).\\
\end{split}\eeq

Since the estimate \eqref{Est-u-H1} holds, it is well-known that the
solution of problem \eqref{NS1}--\eqref{NSi} is global and unique
provided $m>1$.

Theorem \ref{CSR12-Thm-u} is proved.


\section*{Acknowledgments}
This work is supported by National Natural Science Foundation of China--NSF,
Grant No.11971068 and No.11971077.


\end{document}